\documentclass[12pt]{article}

\usepackage[margin=1in]{geometry}  
\usepackage{graphicx}              
\usepackage{amsmath}               
\usepackage{amsfonts}              
\usepackage{amsthm}                
\usepackage[ansinew]{inputenc}
\usepackage{array}
\usepackage{color}
\usepackage{amsmath}
\usepackage{amsxtra}
\usepackage{amstext}
\usepackage{amssymb}
\usepackage{latexsym}
\usepackage{dsfont}
\usepackage{graphicx}
\usepackage[all]{xy}
\usepackage{authblk}
\newtheorem{thm}{Theorem}[section]

\newtheorem{lem}[thm]{Lemma}

\theoremstyle{definition}
\newtheorem{defn}[thm]{Definition}
\theoremstyle{remark}

\theoremstyle{example}
\newtheorem{exm}[thm]{Example}
\numberwithin{equation}{section}

\begin{document}
\title{Enveloping algebras of color hom-Lie algebras}

\author[1]{A. R. Armakan}
\author[2]{S. Silvestrov\thanks{Corresponding author: E-mail: sergei.silvestrov@mdh.se}}
\author[1]{M. R. Farhangdoost}

\affil[1]{Department of Mathematics, College of Sciences, Shiraz University, P.O. Box 71457-44776, Shiraz, Iran.}
\affil[2]{Division of Applied Mathematics,
School of Education, Culture and Communication,
 M\"{a}lardalen University, Box 883, 72123 V{\"a}ster{\aa}s, Sweden.}
\date{}

\maketitle
\abstract{In this paper the universal enveloping algebra of color hom-Lie algebras is studied. A construction of the free involutive hom-associative color algebra on a hom-module is described and applied to obtain the universal enveloping algebra of an involutive hom-Lie color algebra. Finally, the construction is applied to obtain the
well-known Poincar{\'e}-Birkhoff-Witt theorem for Lie algebras to the enveloping algebra of an involutive color hom-Lie algebra.}
\\


\textbf{M.S.C. 2010:}17B75, 17B99, 17B35, 17B37, 17D99, 16W10, 16W55.

\section{Introduction} \label{sec:Intro}
The investigations of various quantum deformations (or $q$-deformations) of Lie algebras began a period of rapid expansion in 1980's stimulated by introduction of quantum groups motivated by applications to the quantum Yang-Baxter equation, quantum inverse scattering methods and constructions of the quantum deformations of universal enveloping algebras of semi-simple Lie algebras.
Since then several other versions of $q$-deformed Lie algebras have appeared, especially in physical contexts such as string theory, vertex models in conformal field theory, quantum mechanics and quantum field theory in the context of deformations of infinite-dimensional algebras, primarily the Heisenberg algebras, oscillator algebras and Witt and Virasoro algebras
\cite{AizawaSaito,ChaiElinPop,ChaiIsLukPopPresn,ChaiKuLuk,ChaiPopPres,CurtrZachos1,DamKu,
DaskaloyannisGendefVir,Hu,Kassel92,LiuKQuantumCentExt,LiuKQCharQuantWittAlg,LiuKQPhDthesis}. In these pioneering works it has been in particular descovered that in these $q$-deformations of Witt and Visaroro algebras and some related algebras, some interesting $q$-deformations of Jacobi identities, extanding Jacobi identity for Lie algebras, are satisfied. This has been one of the initial motivations for the development of general quasi-deformations and discretizations of Lie algebras of vector fields using more general $\sigma$-derivations (twisted derivations) in \cite{HLS}, and introduction of abstract quasi-Lie algebras and subclasses of quasi-Hom-Lie algebras and Hom-Lie algebras as well as their general colored (graded) counterparts in \cite{HLS,LS1,LSGradedquasiLiealg,Czech:witt,LS2}. These generalized Lie algebra structures with (graded) twisted  skew-symmetry and twisted Jacobi conditions by linear maps are tailored to encompass within the same algebraic framework such quasi-deformations and discretizations of Lie algebras of vector fields using $\sigma$-derivations, describing general descritizations and deformations of derivations with twisted Leibniz rule, and the well-known generalizations of Lie algebras such as color Lie algebras which are the natural generalizations of Lie algebras and Lie superalgebras.

Quasi-Lie algebras are non-associative algebras for which the skew-symmetry and the Jacobi identity are twisted by several deforming twisting maps and also the Jacobi identity in quasi-Lie and quasi-Hom-Lie algebras in general contains six twisted triple bracket terms.
Hom-Lie algebras is a special class of quasi-Lie algebras with the bilinear product satisfying the non-twisted skew-symmetry property as in Lie algebras, whereas the Jacobi identity contains three terms twisted by a single linear map, reducing to the Jacobi identity for ordinary Lie algebras when the linear twisting map is the identity map. Subsequently, hom-Lie admissible algebras have been considered in \cite{MS} where also the hom-associative algebras have been introduced and shown to be hom-Lie admissible natural generalizations of associative algebras corresponding to hom-Lie algebras. In \cite{MS}, moreover several other interesting classes of hom-Lie admissible algebras generalising some non-associative algebras, as well as examples of finite-dimentional hom-Lie algebras have been described. Since these pioneering works \cite{HLS,LS1,LSGradedquasiLiealg,LS2,LS3,MS}, hom-algebra structures have become a popular area with increasing number of publications in various directions.

Hom-Lie algebras, hom-Lie superalgebras and hom-Lie color algebras are important special classes of color ($\Gamma$-graded) quasi-Lie algebras introduced first by Larsson and Silvestrov in \cite{LSGradedquasiLiealg,LS2}. Hom-Lie algebras and hom-Lie superalgebras have been studied further
in different aspects by Makhlouf, Silvestrov, Sheng, Ammar, Yau and other authors \cite{SB,SC,SR,MS,HomDeform,HomHopf,HomAlgHomCoalg,YN,BM,MAK,DER,
Yau:EnvLieAlg,Yau:HomolHomLie,Yau:HomBial,LarssonSigSilvJGLTA2008, RichardSilvestrovJA2008, RichardSilvestrovGLTbdSpringer2009, SigSilvGLTbdSpringer2009, SilvestrovParadigmQLieQhomLie2007},
and hom-Lie color algebras have been considered for example in \cite{YN,CCH,spl,COH}.
In \cite{AmmarMakhloufHomLieSupAlg2010}, the constructions of Hom-Lie and quasi-hom Lie algebras based on twisted discretizations of vector fields \cite{HLS} and Hom-Lie admissible algebras have been extended to Hom-Lie superalgebras, a subclass of graded quasi-Lie algebras \cite{LS2,LSGradedquasiLiealg}.
We also wish to mention that $\mathbb{Z}_3$-graded generalizations of supersymmetry, $\mathbb{Z}_3$-graded algebras, ternary structures and related algebraic models
for classifications of elementary particles and unification problems for interactions, quantum gravity and non-commutative gauge theories \cite{Kerner7,Kerner,Kerner2,Kerner4,Kerner6} also provide interesting examples related to Hom-associative algebras, graded Hom-Lie algebras, twisted differential calculi and $n$-ary Hom-algebra structures. It would be a project of grate interest to extend and apply all the constructions and results in the present paper in the relevant contexts of the articles \cite{Kerner7,AmmarMabroukMakhloufCohomnaryHNLalg2011,AmmarMakhloufHomLieSupAlg2010, ArnlindMakhloufSilvTernHomNambuJMP2010,ArnlindMakhloufSilvnaryHomLieNambuJMP2011,ArnlindKituoniMakhloufSilv3aryCohom,
Kerner,Kerner2,Kerner4,LS2,LSGradedquasiLiealg,MS}.

An important direction with many fundamental open problems in the theory of (color) quasi-Lie algebras and in particular (color) quasi-hom-Lie algebras and (color) hom-Lie algebras is the development of comprehensive fundamental theory, explicit constructions, examples and algorithms for enveloping algebraic structures, expanding the corresponding more developed fundamental theory and constructions for enveloping algebras of Lie algebras, Lie superalgebras and general color Lie algebras \cite{BahturinMikhPetrZaicevIDLSbk92,KharchenkoQLTbook2015,MikhZolotykhCALSbk95,
ScheunertCOH2,ScheunertGLA,ScheunertGTC,ScheunertZHA,ChenPetitOystaeyenCOHCHLA,PiontkovskiSilvestrovC3dCLA}. Several authors have tried to construct the enveloping algebras of hom-Lie algebras. For instance, Yau has constructed the enveloping hom-associative algebra $U_{H \textit{Lie}}(L)$ of a hom-Lie algebra $L$ in \cite{Yau:EnvLieAlg} as the left adjoint functor of $H \textit{Lie}$ using combinatorial objects of weighted binary trees, i.e., planar binary trees in which the internal vertices are equipped with weights of non-negative integers. This is analogous to the fact that the functor $\textit{Lie}$ admits a left adjoint $U$, the enveloping algebra functor. He also introduced construction of the counterpart functors $H \textit{Lie}$ and $U_{H \textit{Lie}}$ for hom-Leibniz algebras. In the article \cite{HeMaSiUnAlHomAss}, for hom-associative algebras and hom-Lie algebras, the envelopment problem, operads, and the Diamond Lemma and Hilbert series for the hom-associative operad and free algebra have been studied. Recently, making use of free involutive hom-associative algebras, the authors in \cite{GuoZhZheUEPBWHLieA} have found an explicit constructive way to obtain the universal enveloping algebras of hom-Lie algebras, in order to prove the Poincare Birkhoff Witt theorem.

In this paper we will give a brief review of well-known facts about hom-Lie algebras and their enveloping algebras. We will then present some new results, in the hope that they may eventually have a bearing on representation theory of color hom-Lie algebras. In Section \ref{sec:basics HomLie}, some necessary notions and definitions are presented as an introduction to color hom-Lie algebras. In Section \ref{deff} the notion of universal enveloping algebras of color hom-Lie algebras is given and several useful result about involutive color hom-Lie algebras are proven. In Section \ref{sec:PBWinvhomLie}, we prove the analogous of the well known Poincare-Birkhoff-Witt theorem for color hom-Lie algebras, using the definitions and results of Sections \ref{sec:basics HomLie} and \ref{deff}. Finally, due to the importance of hom-Lie superalgebras, we present the most important results of the paper in hom-Lie superalgebras case.

\section{Basic conseps on hom-Lie algebras and color quasi-Lie algebras} \label{sec:basics HomLie}
We start by recalling some basic concepts from \cite{HomAlgHomCoalg,MS,GuoZhZheUEPBWHLieA}.
We use $\mathbf{k}$ to denote a commutative unital ring (for example a field).
\begin{defn}
\begin{itemize}
\item[(i)] A hom-module is a pair $(M,\alpha)$ consisting of an $\mathbf{k}$-module $M$ and a linear operator $\alpha:M\rightarrow M$.

\item[(ii)] A hom-associative algebra is a triple $(A,.,\alpha)$ consisting of an $\mathbf{k}$-module $A$, a linear map $\cdot:A\otimes A\rightarrow A$ called the multiplication and a multiplicative linear operator $\alpha:A\rightarrow A$ which satisfies the hom-associativity condition, namely
$$\alpha(x)\cdot(y\cdot z)=(x\cdot y)\cdot\alpha(z),$$
for all $x,y,z\in A$.

\item[(iii)] A hom-associative algebra or a hom-module is called involutive if $\alpha^{2}=id$.

\item[(iv)] Let $(M,\alpha)$ and $(N,\beta)$ be two hom-modules. An $\mathcal{k}$-linear map $f:M\rightarrow N$ is called a morphism of hom-modules if
$$f(\alpha(x))=\beta(f(x)),$$
for all $x\in M$.

\item[(v)] Let $(A,\cdot,\alpha)$ and $(B,\bullet,\beta)$ be two hom-associative algebras. A $\mathbf{k}$-linear map $f:A\rightarrow B$ is called a morphism of hom-associative algebras if
\begin{itemize}
\item[(1)] $f(x\cdot y)=f(x)\bullet f(y),$
\item[(2)] $f(\alpha(x))=\beta(f(x)),$
for all $x,y\in A$.
\end{itemize}

\item[(vi)] If $(A,\cdot,\alpha)$ is a hom-associative algebra, then $B\subseteq A$ is called a hom-associative subalgebra of $A$ if it is closed under the multiplication $\cdot$ and $\alpha(B)\subseteq B$. A submodule $I$ is called a hom-ideal of A if $x\cdot y \in I$ and $x\cdot y \in I$ for all $x\in I$ and $y \in A$, and also $\alpha(I)\subseteq I$.

\end{itemize}
\end{defn}
One can find various examples of hom-associative algebras and their properties in \cite{HomAlgHomCoalg,MS}.

\begin{defn}\label{FIHAA}
Let $(M,\alpha)$ be an involutive hom-module. A free involutive hom-associative algebra on $M$ is an involutive hom associative algebra $(F_{M},\star,\beta)$ together with a morphism of hom-modules $j:M\rightarrow F_{M}$ with the property that for any involutive hom-associative algebra $A$ together with a morphism $f:M\rightarrow A$ of hom-modules, there is a unique morphism $\bar{f}:F_{M}\rightarrow A$ of hom-associative algebras such that $f=\bar{f}\circ j$.
\end{defn}

Our next goal is to recall the definition of an involutive hom-associative algebra on an involutive hom-module $(M,\alpha)$ from \cite{GuoZhZheUEPBWHLieA}, which is known as the hom-tensor algebra and is denoted here by $T(M)$. Note that as an $R$-module, $T(M)$ is the same as the tensor algebra, i.e.

$$T(M)=\bigoplus_{i\geq1}M^{\otimes i},$$
on which we have the following multiplication in order to obtain a hom-associative algebra. First, the linear map $\alpha$ on $M$ is extended to a linear map $\alpha_{T}$ on $M^{\otimes i}$ by the tensor multiplicativity, i.e.
$$\alpha_{T}(x)=\alpha_{T}(x_{1}\otimes\cdots\otimes x_{i}):= \alpha(x_{1})\otimes\cdots \otimes \alpha(x_{i}),$$
for all pure tensors $x:=x_{1}\otimes\cdots\otimes x_{i}\in M^{\otimes i}$, $i\geq1$. One can see that $\alpha_{T}$ has the following properties:
\begin{itemize}
\item[(i)] $\alpha_{T}(x\otimes y)=\alpha_{T}(x)\otimes \alpha_{T}(y)$, for all $x\in M^{\otimes i}$ , $y\in M^{\otimes j}$.
\item[(ii)] $\alpha_{T}^{2}=id$.
\end{itemize}
Now, the binary operation on $T(M)$ is defined as follows:

\begin{equation}
x\odot y:=\alpha_{T}^{i-1}(x)\otimes y_{1}\otimes \cdots \alpha_{T}(y_{2}\otimes \cdots \otimes y_{i}),
\end{equation}
for all $x\in M^{\otimes i}$ and $y\in M^{\otimes j}$.

\begin{thm}\cite{GuoZhZheUEPBWHLieA}\label{sakht}
Let $(M,\alpha)$ be an involutive hom-module. Then
\begin{itemize}
  \item [(i)] The triple $T(M):=(T(M),\odot,\alpha_{T})$ is an involutive hom-associative algebra.
  \item [(ii)] The quadruple $(T(M),\odot,\alpha_{T},i_{M})$ is the free involutive hom-associative algebra on $M$.
\end{itemize}
\end{thm}

It is now convenient to recall some definitions for hom-Lie algebras \cite{HLS,LS1,LS2,MS,LSGradedquasiLiealg}.
\begin{defn}
A hom-Lie algebra is a triple $(\mathfrak{g},[,], \alpha)$, where $\mathfrak{g} $ is a vector space equipped with a skew-symmetric bilinear map $[,]: \mathfrak{g} \times \mathfrak{g} \rightarrow \mathfrak{g}$ and a linear map $\alpha :\mathfrak{g} \rightarrow \mathfrak{g} $ such that
$$[\alpha(x),[y,z]]+[\alpha(y),[z,x]]+[\alpha(z),[x,y]]=0, $$
for all $x,y,z \in \mathfrak{g}$ , which is called hom-Jacobi identity.
\end{defn}

A hom-Lie algebra is called a multiplicative hom-Lie algebra if $\alpha$ is an algebraic morphism, i.e. for any $x,y\in \mathfrak{g}$,
$$\alpha([x,y])=[\alpha(x),\alpha(y)].$$

We call a hom-Lie algebra regular if $\alpha$ is an automorphism. Moreover, it is called involutive if $\alpha^{2}=id$.

A sub-vector space $\mathfrak{h}\subset \mathfrak{g}$ is a hom-Lie sub-algebra of $(\mathfrak{g},[,], \alpha)$ if $\alpha(\mathfrak{h})\subset \mathfrak{h}$ and $\mathfrak{h}$ is closed under the bracket operation, i.e.
$$[x_{1},x_{2}]_{\mathfrak{g}}\in \mathfrak{h},$$
for all $x_{1},x_{2}\in \mathfrak{h}.$
Let $(\mathfrak{g},[,], \alpha)$ be a multiplicative hom-Lie algebra. Let $\alpha^{k}$ denote the $k$-times composition of $\alpha$ by itself, for any nonnegative integer $k$, i.e.
$$\alpha^{k}=\alpha \circ ... \circ \alpha ~~~~~(k-times),$$
where we define $\alpha^{0}=Id$ and $\alpha^{1}=\alpha$. For a regular hom-Lie algebra $\mathfrak{g}$, let
$$\alpha^{-k}=\alpha^{-1} \circ ... \circ \alpha^{-1}~~~~~(k-times).$$

We now recall the notion of a hom-Lie color algebra step by step in order to indicate them as a generalization of Lie color algebras.
\begin{defn} \cite{ScheunertGLA,ScheunertGTC,BahturinMikhPetrZaicevIDLSbk92,MikhZolotykhCALSbk95,LSGradedquasiLiealg}
Given a commutative group $\Gamma$ (referred to as the grading group), a commutation factor on $\Gamma$ with values in the multiplicative group $K\setminus \{0\}$ of a field $K$ of characteristic 0 is a map
$$\varepsilon: \Gamma \times \Gamma \rightarrow K\setminus \{0\},$$

satisfying three properties:
\begin{itemize}
  \item [(i)] $\varepsilon(\alpha+\beta,\gamma)=\varepsilon(\alpha,\gamma)\varepsilon(\beta,\gamma),$
  \item [(ii)] $\varepsilon(\alpha,\gamma+\beta)=\varepsilon(\alpha,\gamma)\varepsilon(\alpha,\beta),$
  \item [(iii)] $\varepsilon(\alpha,\beta)\varepsilon(\beta,\alpha)=1.$
\end{itemize}
A $\Gamma$-graded $\varepsilon$-Lie algebra (or a Lie color algebra) is a $\Gamma$-graded linear space
$$X=\bigoplus_{\gamma\in \Gamma} X_{\gamma},$$
with a bilinear multiplication (bracket) $[.,.]:X\times X \rightarrow X$ satisfying the following properties:
\begin{itemize}
\item [(i)] \textbf{Grading axiom:} $[X_{\alpha}, X_{\beta}]\subseteq X_{\alpha+\beta},$
\item [(ii)] \textbf{Graded skew-symmetry:} $[a,b]=-\varepsilon(\alpha,\beta)[b,a],$
\item [(iii)]  \textbf{Generalized Jacobi identity:}  \\$\varepsilon(\gamma,\alpha)[a,[b,c]]+ \varepsilon(\beta,\gamma)[c,[a,b]]+ \varepsilon(\alpha,\beta)[b,[c,a]]=0,$
\end{itemize}
for all $a\in X_{\alpha}, b\in X_{\beta}, c\in X_{\gamma}$ and $\alpha,\beta,\gamma \in \Gamma$.
The elements of $X_{\gamma}$ are called homogenous of degree $\gamma$, for all $\gamma\in \Gamma$.
\end{defn}

Analogous to the other kinds of definitions of hom-algebras, the definition of a color hom-Lie algebra can be given as follows \cite{COH,LSGradedquasiLiealg,YN,CCH,spl}.

\begin{defn} \label{HLCAD}
A color hom-Lie algebra is a quadruple $(\mathfrak{g},[.,.],\varepsilon,\alpha)$ consisting of a $\Gamma$-graded vector space $\mathfrak{g}$, a bi-character $\varepsilon$, an even bilinear mapping $$[.,.]:\mathfrak{g}\times \mathfrak{g}\rightarrow \mathfrak{g},$$ (i.e. $[\mathfrak{g}_{a},\mathfrak{g}_{b}]\subseteq \mathfrak{g}_{a+b}$, for all $a,b \in \Gamma$) and an even homomorphism $\alpha:\mathfrak{g}\rightarrow \mathfrak{g}$ such that for homogeneous elements $x,y,z\in \mathfrak{g}$ we have
\begin{itemize}
  \item [1.] \textbf{$\varepsilon$-skew symmetric:} $[x,y]=-\varepsilon(x,y)[y,x],$
  \item [2.] \textbf{$\varepsilon$-Hom-Jacobi identity:} $\sum_{cyclic\{x,y,z\}}\varepsilon(z,x)[\alpha(x),[y,z]]=0.$
\end{itemize}
\end{defn}

Color hom-Lie algebras are a special class of general color quasi-Lie algebras ($\Gamma$-graded quasi-Lie algebras) defined first by Larsson and Silvestrov in \cite{LSGradedquasiLiealg}.

Let $\mathfrak{g}=\bigoplus_{\gamma\in \Gamma}\mathfrak{g}_{\gamma}$ and $\mathfrak{h}=\bigoplus_{\gamma\in \Gamma}\mathfrak{h}_{\gamma}$ be two $\Gamma$-graded color Lie algebras. A linear mapping $f:\mathfrak{g}\rightarrow \mathfrak{h}$ is said to be homogenous of the degree $\mu\in \Gamma$ if
$$f(\mathfrak{g}_{\gamma})\subseteq \mathfrak{h}_{\gamma+\mu},$$
for all $\gamma\in \Gamma$. If in addition, $f$ is homogenous of degree zero, i. e. $$f(\mathfrak{g}_{\gamma})\subseteq \mathfrak{h}_{\gamma},$$ holds for any $\gamma\in \Gamma$, then $f$ is said to be even.

Let $(\mathfrak{g},[,],\varepsilon,\alpha)$ and $(\mathfrak{g}',[,]',\varepsilon',\alpha')$ be two color hom-Lie algebras. A homomorphism of degree zero $f:\mathfrak{g}\rightarrow \mathfrak{g}'$ is said to be a morphism of color hom-Lie algebras if
\begin{itemize}
  \item [1.] $[f(x),f(y)]'=f([x,y])$, for all $x,y \in \mathfrak{g},$
  \item [2.] $f \circ \alpha=\alpha' \circ f.$
\end{itemize}

In particular, if $\alpha$ is a morphism of color Lie algebras, then we call $(\mathfrak{g},[,],\varepsilon,\alpha)$, a multiplicative color hom-Lie algebra.\\

\begin{exm} \cite{COH}
As in case of hom-associative and hom-Lie algebras,
examples of multiplicative color hom-Lie algebras can be constructed for example by the standard method of composing multiplication with algebra morphism.

Let $(\mathfrak{g},[,],\varepsilon)$ be a color Lie algebra and $\alpha$ be a color Lie algebra morphism. Then $(\mathfrak{g},[,]_{\alpha}:=\alpha\circ [.,.],\varepsilon,\alpha)$ is a multiplicative hom-Lie color algebra.
\end{exm}

\begin{defn}
A hom-associative color algebra is a triple $(V,\mu,\alpha)$ consisting of a color space $V$, an even bilinear map $\mu:V\times V\rightarrow V$ and an even homomorphism $\alpha:V\rightarrow V$ satisfying
$$\mu(\alpha(x),\mu(y,z))=\mu(\mu(x,y),\alpha(z)),$$
for all $x,y,z \in V$.
\end{defn}

A hom-associative color algebra or a hom-Lie color algebra is said to be involutive if $\alpha^{2}=id$.

As in the case of an associative algebra and a Lie algebra, a hom-associative color algebra $(V,\mu,\alpha)$ gives a hom-Lie color algebra by antisymmetrization. We denote this hom-Lie color algebra by $(A,[,]_{A},\beta_{A})$, where $\beta_{A}=\alpha$ and $[x,y]_{A}=xy-yx$, for all $x,y\in A$.

\section{The Universal Enveloping Algebra}\label{deff}
In this section we introduce the notion of the universal enveloping algebra of a hom-Lie color algebra. Moreover, we prove a new result on the free involutive hom-associative color algebra on an involutive hom-module.
\begin{defn}
Let $(V,\alpha_{V})$ be an involutive hom-module. A free involutive hom-associative color algebra on $V$ is an involutive hom-associative color algebra $(F(V),*,\alpha_{F})$ together with a morphism of hom-modules
$$j_{V}:(V,\alpha_{V})\rightarrow (F(V),\alpha_{F}),$$
 with the property that, for any involutive hom-associative color algebra $(A,.,\alpha_{A})$ together with a morphism $f:(V,\alpha_{V})\rightarrow (A,\alpha_{A})$ of hom-modules, there is a unique morphism $\overline{f}:(F(V),*,\alpha_{F})\rightarrow (A,.,\alpha_{A})$ of hom-associative color algebras such that $f=\overline{f}\circ j_{V}$.
\end{defn}

\begin{defn}\label{uni}
Let $(\mathfrak{g},[,],\alpha)$ be a hom-Lie color algebra. A universal enveloping hom-associative color algebra of $\mathfrak{g}$ is a hom associative color algebra $$U(\mathfrak{g}):=(U(\mathfrak{g}),\mu_{U},\alpha_{U}),$$ together with a morphism $\varphi_{\mathfrak{g}}:\mathfrak{g}\rightarrow U(\mathfrak{g})$ of color hom-Lie algebras such that for any hom-associative color algebra $(A,\mu,\alpha_{A})$ and any hom-Lie color algebra morphism $\phi:(\mathfrak{g},[,]_{\mathfrak{g}},\beta_{\mathfrak{g}})$, there exists a unique morphism $\bar{\phi}:U(\mathfrak{g}\rightarrow A)$ of hom associative color algebras such that $\bar{\phi}\circ \varphi_{\mathfrak{g}}=\phi$.

\end{defn}

The following lemma shows an easy way to construct the universal algebra when we have an involutive hom-Lie color algebra.

\begin{lem}\label{3tayi}
Let $(\mathfrak{g},[,]_{\mathfrak{g}},\beta_{\mathfrak{g}})$ be an inovolutive hom-Lie color algebra.
\begin{itemize}
  \item [(i)] Let $(A,\cdot,\alpha_{A})$ be a hom-associative algebra. Let $$f:(\mathfrak{g},[,]_{\mathfrak{g}}, \beta_{\mathfrak{g}})\rightarrow (A,[,]_{A},\beta_{A})$$ be a morphism of color hom-Lie algebras and let $B$ be the hom-associative subcolor algebra of $A$ generated by $f(\mathfrak{g})$. Then $B$ is involutive.
  \item [(ii)] The universal enveloping hom-associative algebra $(U(\mathfrak{g}),\varphi_{\mathfrak{g}})$ of $(\mathfrak{g},[,]_{\mathfrak{g}},\beta_{\mathfrak{g}})$ is involutive.
  \item [(iii)] In order to verify the universal property of $(U(\mathfrak{g}),\varphi_{\mathfrak{g}})$ in Definition \ref{uni}, we only need to consider involutive hom-associative algebras $A:=(A,\cdot_{A},\alpha_{A})$.
\end{itemize}
\end{lem}
\begin{proof}
\begin{itemize}
  \item [\textit{(i)}] Let
  $$S=\{x\in A|\alpha_{A}^{2}(x)=x\}.$$
  One can easily check that $S$ is a submodule. Also for $x,y\in S$, we have $xy\in S$, since $\alpha_{A}^{2}(xy)=\alpha_{A}^{2}(x)\alpha_{A}^{2}(y)=xy$. Moreover, we have
  $$\alpha_{A}^{2}(\alpha_{A}(x))=\alpha_{A} (\alpha_{A}^{2}(x))=\alpha_{A}(x)$$
   which shows that $\alpha_{A}(x)\in S$, for all $x\in S$. Thus, $S$ is a hom-associative subalgebra of $A$. Since $f$ is a morphism of
  \item [\textit{(ii)}] Since $U(\mathfrak{g})$ is generated by $\varphi_{\mathfrak{g}}(\mathfrak{g})$ as a hom associative algebra, the statement follows from $(i)$.
  \item [\textit{(iii)}] We should prove that, assuming that the universal property of $U(\mathfrak{g})$ holds for involutive hom-associative algebras, then it holds for all hom-associative algebras. Let $(A,\cdot,\alpha_{A})$ be a hom-associative algebra and let
      $$\psi: (\mathfrak{g}, [,]_{\mathfrak{g}}, \beta_{\mathfrak{g}}) \rightarrow(A,[,]_{A},\beta_{A})$$
      be a morphism of color hom-Lie algebras. Let
      $$S=\{x\in A|\alpha_{A}^{2}(x)=x\}$$
      be the involutive hom-associative subalgebra of $A$ defined in the proof of $(i)$. Since $im(\psi)$ is contained in $S$, $\psi$ is the composition of a morphism $\psi_{S}: (\mathfrak{g}, [,]_{\mathfrak{g}}, \beta_{\mathfrak{g}}) \rightarrow (S, [,]_{S}, \beta_{S})$ of hom-associative color algebras with the inclusion $i:S\rightarrow A$. By assumption, there is a morphism $\bar{\psi_{S}}:U(\mathfrak{g})\rightarrow A$ of hom-associative color algebras such that $\hat{\psi_{S}}\circ \varphi_{\mathfrak{g}}=\psi_{S}$. Then composing with the inclusion $i:B \rightarrow A$, we obtain a morphism $\bar{\psi}:U(\mathfrak{g})\rightarrow A$ of hom-associative color algebras such that $\bar{\psi} \circ \varphi_{\mathfrak{g}}= \psi$.\\
      Now, let $\bar{\psi}':U(\mathfrak{g})\rightarrow A$ be another morphism of hom-associative color algebras such that $\bar{\psi}' \circ \varphi_{\mathfrak{g}}= \psi$. By $(ii)$, $im(\bar{\psi}')$ is involutive. So $\bar{\psi}'$ is the composition of a morphism $\bar{\psi}'_{S}:U(\mathfrak{g})\rightarrow S$ with the inclusion $i: S\rightarrow A$ and $\bar{\psi}'_{S} \circ \varphi_{\mathfrak{g}}= \psi_{S}$. Since $S$ is involutive, the morphisms $\bar{\psi}'_{S}$ and $\bar{\psi}_{S}$ coincide. As a consequence, $\bar{\psi}'$ and $\bar{\psi}$ coincide which completes the proof.
\end{itemize}
\end{proof}

We can now, give the construction of the universal enveloping hom-associative color algebra of an involutive hom-Lie color algebra.

\begin{thm}\label{I}
Let $\mathfrak{g}:= (\mathfrak{g}, [,]_{\mathfrak{g}},\beta_{\mathfrak{g}})$ be an involutive hom-Lie color algebra. Let $$T(\mathfrak{g}):=(T(\mathfrak{g}),\odot,\alpha_{T})$$ be the free hom-associative algebra on the hom-module underlying $\mathfrak{g}$ obtained in Theorem \ref{sakht}. Let $I$ be the hom-ideal of $T(\mathfrak{g})$ generated by the set
\begin{equation}\label{21}
  \{a\otimes b -\varepsilon(a,b) b\otimes a -[a,b]\}
\end{equation}
and let
$$U(\mathfrak{g})=\frac{T(\mathfrak{g})}{I}$$
be the quotient hom-associative algebra. Let $\psi$ be the composition of the natural inclusion $i:\mathfrak{g}\rightarrow T(\mathfrak{g})$ with the quotient map $\pi:T(\mathfrak{g})\rightarrow U(\mathfrak{g})$. Then $(U(\mathfrak{g}),\psi)$ is a universal enveloping hom-associative algebra of $\mathfrak{g}$. Also, the universal
enveloping hom-associative algebra of $\mathfrak{g}$ is unique up to isomorphism.
\end{thm}

\begin{proof}
Denote by $*$, the multiplication in $U(\mathfrak{g})$. The map $\psi$ is a morphism of hom-modules since it is the composition of two hom-module morphisms. We have
\begin{align*}
  \psi([x,y]_{\mathfrak{g}}) =&\pi([x,y]_{\mathfrak{g}}) = \pi(x\otimes y- \varepsilon(x,y) y\otimes x) \\
   =& \pi(x\odot y- \varepsilon(x,y)y\odot x)= \pi(x)*\pi(y)-\varepsilon(x,y)\pi(y)*\pi(x) \\
  = & \psi(x)*\psi(y)-\varepsilon(x,y)\psi(y)*\psi(x) =[\psi(x),\psi(y)]_{\mathfrak{g}},
\end{align*}
for all $x,y\in \mathfrak{g}$, since $x\otimes y- \varepsilon(x,y) y\otimes x -[x,y]_{\mathfrak{g}}$ is in $I=ker(\pi)$. Therefore, $\psi$ is a morphism of color hom-Lie algebras.
Now, using Lemma\ref{3tayi} $(iii)$, we consider an arbitrary involutive hom-associative color algebra $A:=(A,\cdot_{A}, \alpha_{A})$. Let $\xi:(\mathfrak{g}, [,]_{\mathfrak{g}},\beta_{\mathfrak{g}}) \rightarrow (A,[,]_{A},\beta_{A})$ be a morphism of color hom-Lie algebras. Since $T(\mathfrak{g})$ is the free involutive hom-associative color algebra on the underlying hom-module of $\mathfrak{g}$, according to Theorem \ref{sakht} there exists a hom-associative color algebra morphism $\tilde{\xi}:T(\mathfrak{g})\rightarrow A$ of hom-associative color algebras such that $\tilde{\xi}\circ i_{\mathfrak{g}}=\xi$. We have
\begin{align*}
  \tilde{\xi}(x\otimes y- \varepsilon(x,y) y\otimes x)= &  \tilde{\xi}(x\odot y- \varepsilon(x,y) y\odot x)\\=&\tilde{\xi}(x)\cdot_{A}\tilde{\xi}(y) -\varepsilon(x,y) \tilde{\xi}(y)\cdot_{A}\tilde{\xi}(x)  \\
  = &\xi(x)\cdot_{A}\xi(y) -\varepsilon(x,y)\xi(y) \cdot_{A} \xi(x) =[\xi(x),\xi(y)]_{A} \\
  = & \xi([x,y]_{A})=\tilde{\xi}([x,y]_{A}).
\end{align*}
So $I$ is contained in $ker(\tilde{\xi})$ and $\tilde{\xi}$ induces a morphism $\bar{\xi}:U(\mathfrak{g})\rightarrow A$ of hom-associative color algebras such that $\tilde{\xi}=\bar{\xi}\circ \pi$. Therefore $\bar{\xi}\circ \psi=\bar{\xi}\circ \pi\circ i=\tilde{\xi}\circ i =\xi$.

Let $\bar{\xi}':U(\mathfrak{g})\rightarrow A$ be another morphism of hom-associative color algebras such that $\bar{\xi}'\circ \psi= \xi$. We must show that $\bar{\xi}(u)=\bar{\xi}'(u)$, for all $u\in U(\mathfrak{g})$. It is sufficient to show that  $\bar{\xi}\pi(\mathfrak{a})=\bar{\xi}'\pi(\mathfrak{a})$, for $\mathfrak{a}\in \mathfrak{g}^{\otimes i}$ with $i\geq 1$, since $T(\mathfrak{g})=\bigoplus_{i\geq1}\mathfrak{g}^{\otimes i}$. We do this by using the induction on $i\geq1$. For $i=1$, we have
$$(\bar{\xi}\circ \pi)(\mathfrak{a})=(\bar{\xi}\circ \pi\circ i)(\mathfrak{a})=(\tilde{\xi}\circ i)(\mathfrak{a}) =\xi(\mathfrak{a})=(\bar{\xi}'\circ \psi)(\mathfrak{a})=(\bar{\xi}'\circ \pi)(\mathfrak{a}).$$
Now, assume that the statement holds for $i\geq1$. Let $\mathfrak{a}=\mathfrak{a}'\otimes a_{i+1}\in \mathfrak{g}^{\otimes (i+1)}$, where $\mathfrak{a}'\in \mathfrak{g}^{\otimes i}$. We have
\begin{align*}
  (\bar{\xi}\circ \pi)(\mathfrak{a})= &(\bar{\xi}\circ \pi)(\mathfrak{a}'\odot a_{i+1})=\bar{\xi}(\pi(\mathfrak{a}'))\cdot_{A} \bar{\xi}(\pi( a_{i+1})) \\
  = &  \bar{\xi}'(\pi(\mathfrak{a}'))\cdot_{A} \bar{\xi}'(\pi( a_{i+1}))=(\bar{\xi}'\circ \pi)(\mathfrak{a}'\odot a_{i+1})=(\bar{\xi}'\circ \pi)(\mathfrak{a}).
\end{align*}
Thus, the uniqueness of $\bar{\xi}$ makes $(U(\mathfrak{g}),\psi)$ a universal enveloping algebra of $\mathfrak{g}$.

The only thing left to prove is the uniqueness of $U(\mathfrak{g})$ up to isomorphism. Let $(U(\mathfrak{g})_{1},\psi_{1})$ be another universal enveloping algebra of $\mathfrak{g}$. By the definition of universal algebra, there exist homomorphisms
$$f:U(\mathfrak{g})\rightarrow U(\mathfrak{g})_{1}$$
 and
 $$f_{1}:U(\mathfrak{g})_{1}\rightarrow U(\mathfrak{g})$$
  of hom-associative color algebras such that $f\circ \psi=\psi_{1}$ and $f_{1}\circ \psi_{1}=\psi$. Therefore,
$$f_{1}\circ f\circ \psi=\psi=id_{U(\mathfrak{g})}\circ \psi .$$
Since $id$ and $f_{1}\circ f$ are both hom-associative homomorphisms, by the uniqueness in the universal property of $(U(\mathfrak{g}),\psi)$, we have $f_{1}\circ f=id_{U(\mathfrak{g})}$, and hence
$$f\circ f_{1}=id_{U(\mathfrak{g})_{1}},$$
which completes the proof.
\end{proof}

\section{The Poincare-Birkhoff-Witt Theorem} \label{sec:PBWinvhomLie}
In this section we prove a Poincare-Birkhoff-Witt like type theorem for involutive color hom-Lie algebras.

Let $\mathfrak{g}$ be a Lie color algebra with an ordered basis
$$X=\{x_{n}|n\in H\},$$
where $H$ is a well totally ordered set. Let $I$ be the ideal of the free associative algebra $T(\mathfrak{g})$ on $\mathfrak{g}$ which was given in Theorem \ref{I}, so that $U(\mathfrak{g})$ is the universal enveloping algebra of $\mathfrak{g}$. The Poincare-Birkhoff-Witt Theorem states that the linear subspace $I$ of $T(\mathfrak{g})$ has a canonical linear complement which has a basis given by
\begin{equation}\label{tfW}
W:=\{x_{n_{1}}\otimes \cdots \otimes x_{n_{i}}|n_{1}\geq \cdots \geq n_{i}, i\geq0\},
\end{equation}
called the Poincare-Birkhoff-Witt basis of $U(\mathfrak{g})$.

To simplify the notations, we denote $\bar{x}:=\beta_{\mathfrak{g}}(x)$ for $x\in \mathfrak{g}$ and $\bar{\mathfrak{x}}:=\alpha_{T}(\mathfrak{x})$ for $\mathfrak{x}\in \mathfrak{g}^{\otimes i}$, $i\geq1$. There exists a linear operator which is introduced in \cite{GuoZhZheUEPBWHLieA}.
\begin{equation}\label{teta}
\theta:\mathfrak{g}^{\otimes i} \rightarrow  \mathfrak{g}^{\otimes i},
\end{equation}
which maps every $x:=x_{1}\otimes \cdots \otimes x_{i}$ to $\tilde{x_{1}}\otimes \cdots \otimes \tilde{x_{i}}$, where
\begin{displaymath}
\tilde{x_{n}}= \left\{ \begin{array}{ll}
\bar{x_{n}} & \textrm{if $n=2k+1$ and $k\geq1$,}\\
x_{n} & \textrm{otherwise}
\end{array} \right.
\end{displaymath}

Now, we can study linear generators of the hom-ideal $I$ and express them in terms of the tensor product. Since $\beta_{\mathfrak{g}}$ is involutive, it is also bijective, so
$$\beta_{\mathfrak{g}}(\mathfrak{g})=\mathfrak{g},$$
and we have the same argument for $\theta$:
$$\theta(\mathfrak{g}^{\otimes i})=\mathfrak{g}^{\otimes i}.$$

Let us now review some properties of the linear operator $\alpha_{T}$ and the multiplication $\odot$ which has been stated in \cite{GuoZhZheUEPBWHLieA}. First we have
$$\alpha_{T}^{j}(\mathfrak{g}^{\otimes i})=\mathfrak{g}^{\otimes i}, j\geq0, i\geq0.$$
Then for any natural numbers $r,s\geq1$, we have
$$\mathfrak{g}^{\otimes r}\odot \mathfrak{g}^{\otimes s}= \alpha_{T}^{s-1} (\mathfrak{g}^{\otimes r}) \otimes \mathfrak{g} \otimes \alpha_{T} (\mathfrak{g}^{\otimes (s-1)})= \mathfrak{g}^{\otimes r+s}. $$
In the following lemma, for $\mathfrak{a}=a_{1}\otimes \cdots \otimes a_{i}\in \mathfrak{g}^{\otimes i}$ and $i\geq1$, we denote $\mathfrak{l}(\mathfrak{a})=i$.

\begin{lem}\label{4ta}
Let $\mathfrak{u}=u_{1}\otimes \cdots \otimes u_{j}\in \mathfrak{g}^{\otimes j}$, $\mathfrak{v}=v_{1}\otimes \cdots \otimes v_{k}\in \mathfrak{g}^{\otimes k}$ and $\mathfrak{w}=w_{1}\otimes \cdots \otimes w_{l}\in \mathfrak{g}^{\otimes l}$. Let $\theta$ be defined as in \eqref{teta}. Then
\begin{itemize}
  \item [(i)] $\theta(\mathfrak{u})=u_{1}\otimes u_{2}\otimes \alpha_{T}(u_{3})\otimes \alpha_{T}^{2}(u_{4}) \cdots \otimes \alpha_{T}^{j-2}(u_{j})=u_{1}\otimes \otimes_{k=2}^{j}\alpha_{T}^{k-2}(u_{k})$,
  \item [(ii)] $\theta(\alpha_{T}(\mathfrak{u}))= \alpha_{T}(\theta(\mathfrak{u}))$,
  \item [(iii)] $\theta(\mathfrak{u}\otimes \mathfrak{w})=\theta(\mathfrak{u})\otimes \alpha_{T}^{j-1}(w_{1})\otimes \cdots \otimes \alpha_{T}^{j+l-2}(w_{l}) $
  \item [(iv)] If $\mathfrak{l}(\mathfrak{v})\geq1$ and $\mathfrak{l}(\mathfrak{u})=\mathfrak{l}(\mathfrak{v})+1$, i.e. $j=k+1$, then there is $\mathfrak{c}\in \mathfrak{g}^{\otimes j}$ such that

      $$\theta(\mathfrak{u}\otimes \mathfrak{w})=\theta(\mathfrak{u})\otimes \mathfrak{c} \ \text{ and } \
      \theta(\mathfrak{v}\otimes \mathfrak{w})=\theta(\mathfrak{v})\otimes \alpha_{T}(\mathfrak{c}).$$
\end{itemize}
\end{lem}
\begin{proof}
Straightforward calculations.
\end{proof}

\begin{thm}
Let $\mathfrak{g}$ be an involutive hom-Lie color algebra. Let $\theta:T(\mathfrak{g})\rightarrow T(\mathfrak{g})$ be as defined in \eqref{teta}. Let $I$ be the hom-ideal of $T(\mathfrak{g})$ as defined in Theorem \ref{I}. Let
\begin{equation}\label{J1}
J=  \sum_{n,m\geq0}\sum_{\substack {\mathfrak{a}\in \mathfrak{g}^{\otimes n}\\ \mathfrak{b}\in \mathfrak{g}^{\otimes m}}}\sum_{a,b\in \mathfrak{g}}(\mathfrak{a} \otimes (a\otimes b-\varepsilon(a,b)b\otimes a)\otimes \mathfrak{b} - \alpha_{T}(\mathfrak{a})\otimes [a,b]_{\mathfrak{g}}\otimes \alpha_{T}(\mathfrak{b})).
\end{equation}
Then

\begin{itemize}
  \item [(i)]
$
I=  \sum_{n,m\geq0}\sum_{a,b\in \mathfrak{g}}(\mathfrak{g}^{\otimes n} \odot (a\otimes b-\varepsilon(a,b)b\otimes a -[a,b]_{\mathfrak{g}}) )\odot \mathfrak{g}^{\otimes m}
$
  \item [(ii)] $\theta(I)=J$
\end{itemize}

\end{thm}
\begin{proof}
\begin{itemize}

\item [(i)]

Since the hom-ideal $I$ is generated by the elements of the form
$$a\otimes b-\varepsilon(a,b)b\otimes a -[a,b]_{\mathfrak{g}},$$
 for all $a,b \in \mathfrak{g}$, the right hand side is contained in the left hand side. To prove the opposite, we just need to prove that the right hand side is a hom-ideal of $T(\mathfrak{g})$, i.e. it is closed under the left and right multiplication, and the operator $\alpha_{T}$. Therefore, we should check these, one by one.
For any natural number $k\geq0$, we have
\begin{align*}\label{pf}
  &((\mathfrak{g}^{\otimes n} \odot (a\otimes b-\varepsilon(a,b)b\otimes a -[a,b]_{\mathfrak{g}}) )\odot \mathfrak{g}^{\otimes m})\odot \mathfrak{g}^{\otimes k} \\
  = & \alpha_{T}(\mathfrak{g}^{\otimes n} \odot (a\otimes b-\varepsilon(a,b)b\otimes a -[a,b]_{\mathfrak{g}}) )\odot (\mathfrak{g}^{\otimes m}\odot \alpha_{T}(\mathfrak{g}^{\otimes k}))\\ 
  = & \alpha_{T}(\mathfrak{g}^{\otimes n} \odot (a\otimes b-\varepsilon(a,b)b\otimes a -[a,b]_{\mathfrak{g}}) ) \odot \mathfrak{g}^{\otimes m+k} \\
  = & (\alpha_{T}(\mathfrak{g}^{\otimes n}) \odot \alpha_{T}(a\otimes b-\varepsilon(a,b)b\otimes a -[a,b]_{\mathfrak{g}}) )\odot \mathfrak{g}^{\otimes m+k} \\
  = & (\mathfrak{g}^{\otimes n} \odot (\alpha_{T}(a)\otimes \alpha_{T}(b)-\varepsilon(a,b)\alpha_{T}(b)\otimes \alpha_{T}(a) -[\alpha_{T}(a),\alpha_{T}(b)]_{\mathfrak{g}}) )\\
  &\odot \mathfrak{g}^{\otimes m+k},
\end{align*}
which is seen to be contained in
$$\sum_{x,y\in \mathfrak{g}}(\mathfrak{g}^{\otimes n} \odot (x\otimes y-\varepsilon(x,y)y\otimes x -[x,y]_{\mathfrak{g}}) )\odot \mathfrak{g}^{\otimes m+k}.$$
Thus the right hand side is closed under the right multiplication. We can also get
\begin{align*}
&\alpha_{T}((\mathfrak{g}^{\otimes n} \odot (a\otimes b-\varepsilon(a,b)b\otimes a -[a,b]_{\mathfrak{g}}) ) \odot \mathfrak{g}^{\otimes m})\\
=&(\mathfrak{g}^{\otimes n} \odot (\alpha_{T}(a)\otimes \alpha_{T}(b)-\varepsilon(a,b)\alpha_{T}(b)\otimes \alpha_{T}(a) -[\alpha_{T}(a),\alpha_{T}(b)]_{\mathfrak{g}}) )\\
\odot& \mathfrak{g}^{\otimes m}
\end{align*}
which is contained in
$$\sum_{x,y\in\mathfrak{g}}(\mathfrak{g}^{\otimes n}\odot(x\otimes y-\varepsilon(x,y)y\otimes x -[x,y]_{\mathfrak{g}}))\odot \mathfrak{g}^{\otimes m}.$$
So the right hand side is a hom-ideal of $(T(\mathfrak{g}),\odot,\alpha_{T})$, too which contains the elements of the form
$$(x\otimes y-\varepsilon(x,y)y\otimes x -[x,y]_{\mathfrak{g}}),$$
for all $x,y \in \mathfrak{g}$. Therefore, it contains the left hand side.

\item [(ii)]

We first prove that $\theta(I)$ is contained in $J$. By the first part of the proposition, we just need to verify that the element
$$\theta((\mathfrak{a}\odot (x\otimes y-\varepsilon(x,y)y\otimes x -[x,y]_{\mathfrak{g}}))\odot \mathfrak{b}$$
 s contained in $J$, for $x,y\in \mathfrak{g}$ and $\mathfrak{a},\mathfrak{b}$ in $\mathfrak{g}^{\otimes n},\mathfrak{g}^{\otimes m}$, respectively, $n,m\geq0$.\\
We have
\begin{align*}
   & (\mathfrak{a}\odot (x\otimes y-\varepsilon(x,y)y\otimes x -[x,y]_{\mathfrak{g}}))\odot \mathfrak{b} \\
  = & (\alpha_{T}(\mathfrak{a})\otimes (x\otimes \alpha_{T}(y)-\varepsilon(x,y)y\otimes \alpha_{T}(x)) -\mathfrak{a}\otimes [x,y]_{\mathfrak{g}})\odot \mathfrak{b} \\
  = & \alpha_{T}^{m-1}(\alpha_{T}(\mathfrak{a})\otimes (x\otimes \alpha_{T}(y)-\varepsilon(x,y)y\otimes \alpha_{T}(x))\otimes b_{1}\\
  \otimes & \alpha_{T}(b_{2}\otimes \cdots\otimes b_{m}) \\
  -& \alpha_{T}^{m-1}(\mathfrak{a}\otimes [x,y]_{\mathfrak{g}})\otimes b_{1}\otimes \alpha_{T}(b_{2}\otimes \cdots\otimes b_{m}),
\end{align*}
by the definition of $\odot$. Furthermore, according to lemma \ref{4ta} \textit{(iv)}, there exists $\mathfrak{c}\in \mathfrak{g}^{\otimes n}$ such that
\begin{align*}
   & \theta ((\mathfrak{a}\odot (x\otimes y-\varepsilon(x,y)y\otimes x -[x,y]_{\mathfrak{g}}))\odot \mathfrak{b})\\
  =  &\theta(\alpha_{T}^{m-1} (\alpha_{T}(\mathfrak{a}) \otimes (x\otimes \alpha_{T}(y)-\varepsilon(x,y)y\otimes \alpha_{T}(x))))\otimes _{c} \\
  -& \theta(\alpha_{T}^{m-1}(\mathfrak{a}\otimes [x,y]_{\mathfrak{g}}))\otimes  \alpha_{T}(\mathfrak{c})\\
  = & \theta(\alpha_{T}^{m-1}(\mathfrak{a})\otimes (\alpha_{T}^{m-1}(x) \otimes (x\otimes \alpha_{T}(y)-\varepsilon(x,y)y\otimes \alpha_{T}(x))))\otimes \mathfrak{c} \\
  -& \theta(\alpha_{T}^{m-1}(\mathfrak{a})\otimes [\alpha_{T}^{m-1}(x),\alpha_{T}^{m-1}(y)]_{\mathfrak{g}})) \otimes  \alpha_{T}(\mathfrak{c}) \\
  =& \theta(\alpha_{T}^{m}(\mathfrak{a}))\otimes (\alpha_{T}^{n-1}(\alpha_{T}^{m-1}(x))\otimes \alpha_{T}^{n}(\alpha_{T}^{m}(y))-\varepsilon(x,y) \alpha_{T}^{n-1}(\alpha_{T}^{m-1}(y))\\
  & \otimes \alpha_{T}^{n}(\alpha_{T}^{m}(x)))\otimes \mathfrak{c} -  \theta(\alpha_{T}^{m-1}(\mathfrak{a})\otimes [\alpha_{T}^{m-1}(x),\alpha_{T}^{m-1}(y)]_{\mathfrak{g}})) \otimes  \alpha_{T}(\mathfrak{c}) \\
  = & \theta(\alpha_{T}^{m}(\mathfrak{a}))\otimes (\alpha_{T}^{n+m}(x)\otimes \alpha_{T}^{n+m} (y)-\varepsilon(x,y) \alpha_{T}^{n+m}(y)\otimes \alpha_{T}^{n+m}(x))\otimes \mathfrak{c} \\
  - & \alpha_{T}(\theta(\alpha_{T}^{m}(\mathfrak{a})))\otimes [\alpha_{T}^{n+m}(x),\alpha_{T}^{n+m}(y)]_{\mathfrak{g}})) \otimes  \alpha_{T}(\mathfrak{c}).
\end{align*}
This is an element in $$\sum_{\substack {\mathfrak{u}\in \mathfrak{g}^{\otimes n}\\ \mathfrak{v}\in \mathfrak{g}^{\otimes m}}}\sum_{s,t\in \mathfrak{g}}\mathfrak{u}\otimes (s\otimes t-\varepsilon(s,t)t\otimes s)\otimes \mathfrak{v} - \alpha_{T}(\mathfrak{u})\otimes [s,t]_{\mathfrak{g}}\otimes \alpha_{T}(\mathfrak{v})),$$
if we simply take $\mathfrak{u}:=\theta(\alpha_{T}^{m}(\mathfrak{a}))$, $s:=\alpha_{T}^{m+n}(x)$ and $t:=\alpha_{T}^{m+n}(y)$. Therefore, $\theta(I)$ is contained in $J$. Conversely, since $\theta$ and $\alpha_{T}$ are bijective, the above argument shows that any term
$$\mathfrak{u}\otimes (s\otimes t-\varepsilon(s,t)t\otimes s)\otimes \mathfrak{v} - \alpha_{T}(\mathfrak{u})\otimes [s,t]_{\mathfrak{g}}\otimes \alpha_{T}(\mathfrak{v})),$$
in the previous sum can be expressed in the form
$$\theta((\mathfrak{a} \odot (x\otimes y-\varepsilon(x,y)y\otimes x -[x,y]_{\mathfrak{g}}) )\odot \mathfrak{b}),$$
which shows the surjectivity of $\theta$ and completes the proof.

\end{itemize}
\end{proof}
In the next theorem, we suppose that $\mathfrak{g}$ is an involutive hom-Lie color algebra with a basis $X=\{x_{n}|n\in \omega \}$ for a well ordered set $\omega$.

\begin{thm}\label{dec}
Let $\mathfrak{g}:= (\mathfrak{g},[,]_{\mathfrak{g}},\beta_{\mathfrak{g}})$ be an involutive hom-Lie color algebra such that $\beta_{\mathfrak{g}}(X)=X$. Let $W$ be the one defined in \eqref{tfW} and let $\mu\in k$ be given. If we define
\begin{equation}\label{J}
J_{\mu}:=  \sum_{n,m\geq0}\sum_{\substack {\mathfrak{a}\in \mathfrak{g}^{\otimes n}\\ \mathfrak{b}\in \mathfrak{g}^{\otimes m}}}\sum_{a,b\in \mathfrak{g}}(\mathfrak{a}\otimes (a\otimes b-\varepsilon(a,b)b\otimes a)\otimes \mathfrak{b} - \mu^{n+m} \alpha_{T}(\mathfrak{a})\otimes [a,b]_{\mathfrak{g}}\otimes \alpha_{T}(\mathfrak{b})),
\end{equation}
then we can have the linear decomposition
$$T(\mathfrak{g})=J_{\mu}\oplus k W.$$

\end{thm}

\begin{proof}
Let us first introduce some notations. For $i\geq2$ let
$$\mathfrak{x}:=x_{n_{1}}\otimes x_{n_{2}}\otimes \cdots \otimes x_{n_{i}}\in X^{\otimes i}\subseteq \mathfrak{g}^{\otimes i}.$$
Define the index of $\mathfrak{x}$ to be
$$d:=\mid\{(r,s)|r<s, n_{r}< n_{s},1\leq r,s\leq i\}\mid .$$
Let $\mathfrak{g}_{i,d}$ be the linear span of all pure tensors $\mathfrak{x}$ of degree $i$ and index $d$. Then we have
$$\mathfrak{g}^{\otimes i}=\bigoplus_{d\geqslant 0}\mathfrak{g}_{i,d}.$$
In particular, $\mathfrak{g}_{i,0}=kW^{(i)}$, where
$$W^{(i)}:=\{x_{n_{1}}\otimes x_{n_{2}}\otimes \cdots \otimes x_{n_{i}}\in X^{\otimes i}|n_{1}\geqslant n_{2}\geqslant \cdots \geqslant n_{i}\}.$$

In order to prove $T(\mathfrak{g})=J_{\mu}\oplus kW$, we need to prove that
$$\mathfrak{g}^{\otimes i}\subseteq J_{\mu} \oplus \sum_{1\leqslant q\leqslant i} kW^{(q)},$$
using induction on $i\geq 1$. For $i=1$, we have $kW^{(1)}=\mathfrak{g}$. So $\mathfrak{g}\subseteq J_{\mu} \oplus kW^{(1)} $. Suppose that the above equation is true for $n\geq1$. Since $\mathfrak{g}^{\otimes (i+1)}=\sum_{d\geq0}\mathfrak{g}_{i+1,d}$, we just need to prove that
$$\mathfrak{g}_{i+1,d}\subseteq J_{\mu} \oplus \sum_{1\leqslant q\leqslant i+1} kW^{(q)},$$
for all $d\geq0$. We use induction on $d$. For $d=0$, we have $\mathfrak{g}_{i+1,0}=kW^{(i+1)}$. Suppose for $l\geq0$ we have
$$\mathfrak{g}_{i+1,l}\subseteq J_{\mu} \oplus \sum_{1\leqslant q\leqslant i+1} kW^{(q)}.$$
Let $\mathfrak{x}=x_{n_{1}}\otimes x_{n_{2}}\otimes \cdots \otimes x_{n_{i}}\in X^{\otimes(i+1)}\cap \mathfrak{q}_{i+1,l+1}$. Since $l+1\geq1$, we can choose an integer $1\leq r\leq i$ such that $n_{r}\leq n_{r+1} $. Let
$$\mathfrak{x}'=x_{n_{1}}\otimes  \cdots \otimes x_{n_{r+1}}\otimes x_{n_{r}}\otimes \cdots \otimes x_{n_{i+1}}$$
be the pure tensor formed by interchanging $x_{n_{r}}$ by $x_{n_{r+1}}$ in $\mathfrak{x}$. Then
$$\mathfrak{x}'\in \mathfrak{g}_{i+1,l}\subseteq J_{\mu} \oplus \sum_{1\leqslant q\leqslant i+1} kW^{(q)}.$$
Since the definition of $J_{\mu}$ gives
$$\mathfrak{x}-\mathfrak{x}'\equiv \mu^{i-1} \alpha_{T}(x_{n_{1}}\otimes \cdots \otimes x_{n_{i}}) \otimes [x_{n_{r}},x_{n_{r+1}}]_{\mathfrak{g}} \otimes \alpha_{T}(x_{n_{r+2}}\otimes \cdots \otimes x_{n_{i+1}} )(mod ~~J_{\mu}),$$
by the induction hypothesis on $i$, we have
$$\mathfrak{x}\in J_{\mu} \oplus \sum_{1\leqslant q\leqslant i+1} kW^{(q)} \oplus \sum_{1\leqslant q\leqslant i} kW^{(q)}.$$
So $\mathfrak{x}$ is in $J_{\mu} \oplus \sum_{1\leqslant q\leqslant i+1} kW^{(q)}$. This proves that $\mathfrak{g}_{i+1,l+1}\subseteq J_{\mu} \oplus \sum_{1\leqslant q\leqslant i+1} kW^{(q)}$. Hence, $\mathfrak{g}^{\otimes i+1}\subseteq J_{\mu} \oplus \sum_{1\leqslant q\leqslant i+1} kW^{(q)}$ which completes the induction steps on $i$.

Now, we want to show that $J_{\mu}\cap kW=0$. Let $S$ be an operator on $T(\mathfrak{g})$ such that
\begin{eqnarray}  \label{model}
\text{(i)}
&S(t)=t,\text{ for all } t\in W.  \nonumber \\
\text{(ii)}
&\text{if } p\geq 2, 1\leq s\leq p-1,\text{ and } n_{s}<n_{s+1},\text{ then } \\ \nonumber
&S(x_{n_{1}}\otimes  \cdots \otimes x_{n_{s}}\otimes x_{n_{s+1}}\otimes \cdots \otimes x_{n_{p}})\\ \nonumber
&= S(x_{n_{1}}\otimes  \cdots \otimes x_{n_{s+1}}\otimes x_{n_{s}}\otimes \cdots \otimes x_{n_{p}})\\ \nonumber
&+S(\mu^{p-2}\alpha_{T}(x_{n_{1}}\otimes \cdots\otimes x_{n_{s-1}})\otimes [x_{n_{s}},x_{n_{s+1}}]_{\mathfrak{g}}\otimes \alpha_{T}(x_{n_{s+2}}\otimes \cdots \otimes x_{n_{p}})).
\end{eqnarray}
We define $S$ on $\sum_{1\leqslant q\leqslant i}\mathfrak{g}^{\otimes q}$ by induction on $i$. For $i=1$, we define $S:=Id_{\mathfrak{g}} $. Let $n\geq2$ and let $S$ to be an operator on $\sum_{1\leqslant q\leqslant i}\mathfrak{g}^{\otimes q}$ satisfying \eqref{model} for all tensors of degree $i$. Note that $\mathfrak{g}^{\otimes i+1}= \sum_{d\leqslant 0}\mathfrak{g}_{i+1,d}$. For a pure tensor $$\mathfrak{x}= x_{n_{1}}\otimes x_{n_{2}}\otimes \cdots \otimes x_{n_{i+1}}\in X^{\otimes i+1}\subseteq \mathfrak{g}^{\otimes i+1},$$ we use induction again on $d$ which is the index of $\mathfrak{x}$, in order to extend $S$ to an operator on $\sum_{1\leqslant q\leqslant i+1}\mathfrak{g}^{\otimes q}$. For $d=0$, define $S(\mathfrak{x})=\mathfrak{x}$. For $l\geq0$, suppose that $S(\mathfrak{x})$ has been defined for $$\mathfrak{x}\in \sum_{1\leqslant p \leqslant l}\mathfrak{g}_{i+1,p}. $$
Let $\mathfrak{x}\in \mathfrak{g}_{i+1,l+1}$. Let $1\leqslant r\leqslant i$ be an integer such that $n_{r}< n_{r+1}$. Then
\begin{align*}
  & S(\mathfrak{x}) :=  S(x_{n_{1}}\otimes  \cdots \otimes x_{n_{r+1}}\otimes x_{n_{r}}\otimes \cdots \otimes x_{n_{i+1}}) \\
 & +  S(\mu^{i-1}\alpha_{T}(x_{n_{1}}\otimes \cdots\otimes x_{n_{r-1}})\otimes [x_{n_{r}},x_{n_{r+1}}]_{\mathfrak{g}}\otimes \alpha_{T}(x_{n_{r+2}}\otimes \cdots \otimes x_{n_{i+1}})).
\end{align*}
We should show that $S$ is well defined and it is independent of the choice of $r$. Therefore, let $r'$ bee another integer, $1\leq r'\leq i$, such that $n_{r'}< n_{r'+1}$. Consider
\begin{align*}
  u:= & S(x_{n_{1}}\otimes  \cdots \otimes x_{n_{r+1}}\otimes x_{n_{r}}\otimes \cdots \otimes x_{n_{i+1}}) \\
  + & S(\mu^{i-1}\alpha_{T}(x_{n_{1}}\otimes \cdots\otimes x_{n_{r-1}})\otimes [x_{n_{r}},x_{n_{r+1}}]_{\mathfrak{g}}\otimes \alpha_{T}(x_{n_{r+2}}\otimes \cdots \otimes x_{n_{i+1}})),
\end{align*}
and
\begin{align*}
  v:= & S(x_{n_{1}}\otimes  \cdots \otimes x_{n_{r'+1}}\otimes x_{n_{r'}}\otimes \cdots \otimes x_{n_{i+1}}) \\
  + & S(\mu^{i-1}\alpha_{T}(x_{n_{1}}\otimes \cdots\otimes x_{n_{r'-1}})\otimes [x_{n_{r'}},x_{n_{r'+1}}]_{\mathfrak{g}}\otimes \alpha_{T}(x_{n_{r'+2}}\otimes \cdots \otimes x_{n_{i+1}})).
\end{align*}
We check that $u=v$.
There appear two cases

\noindent \textbf{Case 1:} If $|r-r'|\geq2$, without losing the generality, we assume $r-r'\geq2$. Since $u,v \in \sum_{0\leqslant p \leqslant l}\mathfrak{g}_{i+1,p}+\sum_{1\leqslant q \leqslant i}\mathfrak{g}^{\otimes q}$, we have
      \begin{align*}
        u = &  S(x_{n_{1}}\otimes  \cdots \otimes x_{n_{r+1}}\otimes x_{n_{r}}\otimes \cdots \otimes x_{n_{r'}}\otimes x_{n_{r'+1}}\otimes \cdots \otimes x_{n_{i+1}}) \\
        + & S(\mu^{i-1}\alpha_{T}(x_{n_{1}}\otimes \cdots\otimes x_{n_{r-1}})\otimes [x_{n_{r}},x_{n_{r+1}}]_{\mathfrak{g}} \otimes \alpha_{T}(x_{n_{r+2}}\otimes \cdots \otimes x_{n_{i+1}})) \\
        = & S(x_{n_{1}}\otimes  \cdots \otimes x_{n_{r+1}}\otimes x_{n_{r}}\otimes \cdots \otimes x_{n_{i+1}}) \\
        + & S(\mu^{i-1}\alpha_{T}(x_{n_{1}}\otimes \cdots\otimes x_{n_{r-1}}\otimes x_{n_{r}}\otimes \cdots) \otimes [x_{n_{r'}},x_{n_{r'+1}}]_{\mathfrak{g}}\\
        &\otimes \alpha_{T}(\cdots \otimes x_{n_{r+1}})) \\
        + & S(\mu^{i-1}\alpha_{T}(x_{n_{1}}\otimes \cdots)\otimes [x_{n_{r}},x_{n_{r+1}}]_{\mathfrak{g}} \\
        &\otimes \alpha_{T}(\cdots \otimes x_{n_{r'}}\otimes x_{n_{r'+1}}\otimes \cdots \otimes x_{n_{i+1}})),\\
%
        v = &  S(x_{n_{1}}\otimes  \cdots \otimes x_{n_{r}}\otimes x_{n_{r+1}}\otimes \cdots \otimes x_{n_{r'+1}}\otimes x_{n_{r'}}\otimes \cdots \otimes x_{n_{i+1}}) \\
        + & S(\mu^{i-1}\alpha_{T}(x_{n_{1}}\otimes \cdots\otimes x_{n_{r}}\otimes x_{n_{r-1}}\otimes \cdots )\otimes [x_{n_{r'}},x_{n_{r'+1}}]_{\mathfrak{g}}\\
        &\otimes \alpha_{T}(x_{n_{r+2}}\otimes \cdots \otimes x_{n_{i+1}})) \\
        = & S(x_{n_{1}}\otimes \cdots \otimes x_{n_{r+1}}\otimes x_{n_{r}}\otimes \cdots \otimes x_{n_{r'+1}}\otimes x_{n_{r'}}\otimes \cdots\otimes  x_{n_{i+1}}) \\
        + & S(\mu^{i-1}\alpha_{T}(x_{n_{1}}\otimes \cdots) \otimes [x_{n_{r}} ,x_{n_{r+1}}]_{\mathfrak{g}}\\
         &\otimes \alpha_{T}(\cdots \otimes x_{n_{r'+1}}\otimes x_{n_{r'+1}}\otimes x_{n_{i+1}})) \\
        + & S(\mu^{i-1}\alpha_{T}(x_{n_{1}}\otimes \cdots \otimes x_{n_{r}}\otimes x_{n_{r+1}}\otimes \cdots)\otimes [x_{n_{r'}},x_{n_{r'+1}}]_{\mathfrak{g}} \\
        &\otimes \alpha_{T}(\cdots \otimes x_{n_{i+1}})),\\
      \end{align*}
      by the induction hypothesis. Now, since $\beta_{\mathfrak{g}}(X)=X$ and $$x_{n_{r}}\neq x_{n_{r+1}}, ~~~~~x_{n_{r'}}\neq x_{n_{r'+1}},$$ we get $\alpha_{T}(x_{n_{r}}), \alpha_{T}(x_{n_{r+1}}), \alpha_{T}(x_{n_{r'}}), \alpha_{T}(x_{n_{r'+1}})\in X$ and
      $$\alpha_{T}(x_{n_{r}})\neq \alpha_{T}(x_{n_{r+1}}),~~~~~\alpha_{T}(x_{n_{r'}})\neq \alpha_{T}(x_{n_{r'+1}}).$$
      This leads to four different cases among of which we only consider the case of $\alpha_{T}(x_{n_{r}})> \alpha_{T}(x_{n_{r+1}})$ and $\alpha_{T}(x_{n_{r'}})< \alpha_{T}(x_{n_{r'+1}})$ without loosing the generality.
      We obtain
      \begin{align*}
         & S(\mu^{i-1}\alpha_{T}(x_{n_{1}}\otimes \cdots\otimes x_{n_{r+1}}\otimes x_{n_{r}}\otimes \cdots )\otimes [x_{n_{r'}},x_{n_{r'+1}}]_{\mathfrak{g}} \\
         &\otimes \alpha_{T}(\cdots \otimes x_{n_{i+1}})) \\
        & = S(\mu^{i-1}\alpha_{T}(x_{n_{1}}\otimes \cdots\otimes x_{n_{r}}\otimes x_{n_{r+1}}\otimes \cdots )\otimes [x_{n_{r'}},x_{n_{r'+1}}]_{\mathfrak{g}} \\
        &\otimes\alpha_{T}(\cdots \otimes x_{n_{i+1}})) \\
        &+ S(\mu^{2i-3}(x_{n_{1}}\otimes \cdots\otimes [\alpha_{T}(x_{n_{r+1}}), \alpha_{T}(x_{n_{r}})]_{\mathfrak{g}} \\
        &\otimes \cdots \otimes [\alpha_{T}(x_{n_{r'}}), \alpha_{T}(x_{n_{r'+1}})]_{\mathfrak{g}} \otimes \cdots \otimes x_{n_{i+1}})),
      \end{align*}
     and
   \begin{align*}
         & S(\mu^{i-1}\alpha_{T}(x_{n_{1}}\otimes \cdots)\otimes [x_{n_{r}},x_{n_{r+1}}]_{\mathfrak{g}} \\
         &\otimes\alpha_{T}(\cdots \otimes  x_{n_{r'}} \otimes x_{n_{r'+1}}\otimes \cdots\otimes  x_{n_{i+1}})) \\
        = & S(\mu^{i-1}\alpha_{T}(x_{n_{1}}\otimes \cdots)\otimes [x_{n_{r}}, x_{n_{r+1}}]_{\mathfrak{g}} \\
        &\otimes\alpha_{T}(\cdots \otimes  x_{n_{r'+1}} \otimes x_{n_{r'}}\otimes \cdots\otimes  x_{n_{i+1}})) \\
        + & S(\mu^{2i-3}(x_{n_{1}}\otimes \cdots\otimes [\alpha_{T}(x_{n_{r}}), \alpha_{T}(x_{n_{r+1}})]_{\mathfrak{g}}\\
        &\otimes \cdots \otimes [\alpha_{T}(x_{n_{r'}}), \alpha_{T}(x_{n_{r'+1}})]_{\mathfrak{g}} \otimes \cdots \otimes x_{n_{i+1}})).
      \end{align*}
      If we combine the two former expressions for $u$, we get
      \begin{align*}
        u = &  S(x_{n_{1}}\otimes  \cdots \otimes x_{n_{r+1}}\otimes x_{n_{r}}\otimes \cdots \otimes x_{n_{r'+1}}\otimes x_{n_{r'}}\otimes \cdots \otimes x_{n_{i+1}}) \\
        + & S(\mu^{i-1}\alpha_{T}(x_{n_{1}}\otimes \cdots\otimes x_{n_{r}}\otimes x_{n_{r+1}}\otimes \cdots) \otimes [x_{n_{r'}},x_{n_{r'+1}}]_{\mathfrak{g}}\\
        &\otimes \alpha_{T}(\cdots \otimes x_{n_{i+1}})) \\
        + & S(\mu^{i-1}\alpha_{T}(x_{n_{1}}\otimes \cdots)\otimes [x_{n_{r}},x_{n_{r+1}}]_{\mathfrak{g}}\\
         &\otimes \alpha_{T}(\cdots \otimes x_{n_{r'+1}}\otimes x_{n_{r'}}\otimes \cdots \otimes x_{n_{i+1}}))\\
        + & S(\mu^{2i-3}(x_{n_{1}}\otimes \cdots\otimes [\alpha_{T}(x_{n_{r+1}}), \alpha_{T}(x_{n_{r}})]_{\mathfrak{g}}\otimes \cdots \\
        &\otimes [\alpha_{T}(x_{n_{r'}}), \alpha_{T}(x_{n_{r'+1}})]_{\mathfrak{g}} \otimes \cdots \otimes x_{n_{i+1}}))\\
        + & S(\mu^{2i-3}(x_{n_{1}}\otimes \cdots\otimes [\alpha_{T}(x_{n_{r}}), \alpha_{T}(x_{n_{r+1}})]_{\mathfrak{g}}\otimes \cdots \\
        &\otimes [\alpha_{T}(x_{n_{r'}}), \alpha_{T}(x_{n_{r'+1}})]_{\mathfrak{g}} \otimes \cdots \otimes x_{n_{i+1}}))
      \end{align*}
      \begin{align*}
      = & S(x_{n_{1}}\otimes  \cdots \otimes x_{n_{r+1}}\otimes x_{n_{r}}\otimes \cdots \otimes x_{n_{r'+1}}\otimes x_{n_{r'}}\otimes \cdots x_{n_{i+1}})\\
        + & S(\mu^{i-1}\alpha_{T}(x_{n_{1}}\otimes \cdots\otimes x_{n_{r}}\otimes x_{n_{r+1}}\otimes \cdots) \otimes [x_{n_{r'}},x_{n_{r'+1}}]_{\mathfrak{g}}\\
        &\otimes \alpha_{T}(\cdots \otimes x_{n_{i+1}}))\\
        + & S(\mu^{i-1}\alpha_{T}(x_{n_{1}}\otimes \cdots)\otimes [x_{n_{r}}, x_{n_{r+1}}]_{\mathfrak{g}}\\
        &\otimes \alpha_{T}(\cdots \otimes  x_{n_{r'+1}} \otimes x_{n_{r'}}\otimes \cdots\otimes  x_{n_{i+1}}))= v,
      \end{align*}
      by the skew-symmetry condition of the bracket.

\noindent \textbf{Case 2:} If $|r-r'|=1$, ones again, without loosing the generality, let us suppose $r'=r+1$. This leads to $n_{r}<n_{i+1}<n_{i+2}$. We obtain
      \begin{align*}
        u = &  S(x_{n_{1}}\otimes  \cdots \otimes x_{n_{r+1}}\otimes x_{n_{r}}\otimes x_{n_{r+2}} \otimes \cdots \otimes x_{n_{i+1}}) \\
        + & S(\mu^{i-1}\alpha_{T}(x_{n_{1}}\otimes \cdots\otimes x_{n_{r-1}})\otimes [x_{n_{r}},x_{n_{r+1}}]_{\mathfrak{g}}\\
        &\otimes \alpha_{T}(x_{n_{r+2}}\otimes \cdots \otimes x_{n_{i+1}})) \\
        = & S(x_{n_{1}}\otimes \cdots \otimes x_{n_{r+1}}\otimes x_{n_{r+2}}\otimes x_{n_{r}} \otimes \cdots \otimes x_{n_{i+1}})\\
        + & S(\mu^{i-1}\alpha_{T}(x_{n_{1}}\otimes \cdots\otimes x_{n_{r+1}})\otimes [x_{n_{r}},x_{n_{r+2}}]_{\mathfrak{g}}\otimes \alpha_{T}(\cdots \otimes x_{n_{i+1}}))\\
        + & S(\mu^{i-1}\alpha_{T}(x_{n_{1}}\otimes \cdots\otimes x_{n_{r-1}})\otimes [x_{n_{r}},x_{n_{r+1}}]_{\mathfrak{g}}\\
        &\otimes \alpha_{T}(x_{n_{r+2}}\otimes \cdots \otimes x_{n_{i+1}}))\\
        = & S(x_{n_{1}}\otimes \cdots \otimes x_{n_{r+2}}\otimes x_{n_{r+1}}\otimes x_{n_{r}} \otimes \cdots \otimes x_{n_{i+1}})\\
        + & S(\mu^{i-1}\alpha_{T}(x_{n_{1}}\otimes \cdots)\otimes [x_{n_{r+1}}, x_{n_{r+2}}]_{\mathfrak{g}}\otimes \alpha_{T}(x_{n_{r}}\otimes \cdots \otimes x_{n_{i+1}}))\\
        + & S(\mu^{i-1}\alpha_{T}(x_{n_{1}}\otimes \cdots\otimes x_{n_{r+1}})\otimes [x_{n_{r}},x_{n_{r+2}}]_{\mathfrak{g}}\otimes \alpha_{T}(\cdots \otimes x_{n_{i+1}}))\\
        + & S(\mu^{i-1}\alpha_{T}(x_{n_{1}}\otimes \cdots\otimes x_{n_{r-1}})\otimes [x_{n_{r}},x_{n_{r+1}}]_{\mathfrak{g}}\\&\otimes \alpha_{T}(x_{n_{r+2}}\otimes \cdots \otimes x_{n_{i+1}})),
      \end{align*}
      and
      \begin{align*}
        v = &  S(x_{n_{1}}\otimes  \cdots \otimes x_{n_{r}}\otimes x_{n_{r+2}}\otimes x_{n_{r+1}} \otimes \cdots \otimes x_{n_{i+1}}) \\
        + & S(\mu^{i-1}\alpha_{T}(x_{n_{1}}\otimes \cdots\otimes x_{n_{r}})\otimes [x_{n_{r+1}},x_{n_{r+2}}]_{\mathfrak{g}}\otimes \alpha_{T}( \cdots \otimes x_{n_{i+1}})) \\
        = & S(x_{n_{1}}\otimes \cdots \otimes x_{n_{r+2}}\otimes x_{n_{r}}\otimes x_{n_{r+1}} \otimes \cdots \otimes x_{n_{i+1}})\\
        + & S(\mu^{i-1}\alpha_{T}(x_{n_{1}}\otimes \cdots)\otimes [x_{n_{r}}, x_{n_{r+2}}]_{\mathfrak{g}}\otimes \alpha_{T}(x_{n_{r+1}}\otimes \cdots \otimes x_{n_{i+1}}))\\
        + & S(\mu^{i-1}\alpha_{T}(x_{n_{1}}\otimes \cdots\otimes x_{n_{r}})\otimes [x_{n_{r+1}},x_{n_{r+2}}]_{\mathfrak{g}}\otimes \alpha_{T}(\cdots \otimes x_{n_{i+1}}))
      \end{align*}
      \begin{align*}
      = & S(x_{n_{1}}\otimes \cdots \otimes x_{n_{r+2}}\otimes x_{n_{r+1}}\otimes x_{n_{r}} \otimes \cdots \otimes x_{n_{i+1}})\\
        + & S(\mu^{i-1}\alpha_{T}(x_{n_{1}}\otimes \cdots\otimes x_{n_{r+2}} )\otimes [x_{n_{r}}, x_{n_{r+1}}]_{\mathfrak{g}}\otimes \alpha_{T}(\cdots \otimes x_{n_{i+1}}))\\
        + & S(\mu^{i-1}\alpha_{T}(x_{n_{1}}\otimes \cdots)\otimes [x_{n_{r}}, x_{n_{r+2}}]_{\mathfrak{g}}\otimes \alpha_{T}(x_{n_{r+1}} \otimes\cdots \otimes x_{n_{i+1}}))\\
        + & S(\mu^{i-1}\alpha_{T}(x_{n_{1}}\otimes \cdots\otimes x_{n_{r}})\otimes [x_{n_{r+1}},x_{n_{r+2}}]_{\mathfrak{g}}\otimes \alpha_{T}(\cdots \otimes x_{n_{i+1}})).
      \end{align*}
      Moreover, for any $a<b$, $t_{1}\in \mathfrak{g}^{\otimes m}$, $t_{2}\in \mathfrak{g}^{\otimes k}$, $m+k=i-2$, we have
      $$S(t_{1}\otimes a\otimes b\otimes t_{2})-S(t_{1}\otimes b\otimes a\otimes t_{2})=S(\mu^{m+k}\alpha_{T}(t_{1})\otimes [a,b] \otimes \alpha_{T}(t_{2})).$$
      So the sum of the last three terms of the previous expression of $u$ is
      \begin{align*}
         & S(\mu^{i-1}\alpha_{T}(x_{n_{1}}\otimes \cdots)\otimes [x_{n_{r+1}}, x_{n_{r+2}}]_{\mathfrak{g}}\otimes \alpha_{T}(x_{n_{r}} \otimes\cdots \otimes x_{n_{i+1}})) \\
        + & S(\mu^{i-1}\alpha_{T}(x_{n_{1}}\otimes \cdots \otimes x_{n_{r+1}})\otimes [x_{n_{r}}, x_{n_{r+2}}]_{\mathfrak{g}}\otimes \alpha_{T}(\cdots \otimes x_{n_{i+1}})) \\
        + & S(\mu^{i-1}\alpha_{T}(x_{n_{1}}\otimes \cdots \otimes x_{n_{r-1}})\otimes [x_{n_{r}}, x_{n_{r+1}}]_{\mathfrak{g}}\otimes \alpha_{T}(x_{n_{i+2}}\otimes \cdots \otimes x_{n_{i+1}})) \\
        = & S(\mu^{i-1}\alpha_{T}(x_{n_{1}}\otimes \cdots)\otimes\alpha_{T}(x_{n_{r}})\otimes [x_{n_{r+1}}, x_{n_{r+2}}]_{\mathfrak{g}}\otimes \alpha_{T}(\cdots \otimes x_{n_{i+1}})) \\
        + & S(\mu^{2i-3}(x_{n_{1}}\otimes \cdots\otimes [[x_{n_{r+1}}, x_{n_{r+2}}]_{\mathfrak{g}}, \alpha_{T}(x_{n_{r}})]_{\mathfrak{g}} \otimes \cdots \otimes x_{n_{i+1}})) \\
        + & S(\mu^{i-1}\alpha_{T}(x_{n_{1}}\otimes \cdots)\otimes [x_{n_{r}}, x_{n_{r+2}}]_{\mathfrak{g}}\otimes \alpha_{T}(x_{n_{r+1}})\otimes \alpha_{T}(\cdots \otimes x_{n_{i+1}})) \\
        + & S(\mu^{2i-3}x_{n_{1}}\otimes \cdots\otimes x_{n_{r-1}}\otimes [\alpha_{T}(x_{n_{r+1}}), [x_{n_{r}}, x_{n_{r+2}}]_{\mathfrak{g}}]_{\mathfrak{g}} \otimes \cdots \otimes x_{n_{i+1}}) \\
        + & S(\mu^{i-1}\alpha_{T}(x_{n_{1}}\otimes \cdots\otimes x_{n_{r-1}} )\otimes \alpha_{T}(x_{n_{r+2}})\otimes [x_{n_{r}}, x_{n_{r+1}}]_{\mathfrak{g}}\\
        &\otimes \alpha_{T}(\cdots \otimes x_{n_{i+1}})) \\
        + & S(\mu^{2i-3}x_{n_{1}}\otimes \cdots\otimes x_{n_{r-1}}\otimes [[x_{n_{r}}, x_{n_{r+1}}]_{\mathfrak{g}}, \alpha_{T}(x_{n_{r+2}})]_{\mathfrak{g}} \otimes \cdots \otimes x_{n_{i+1}}) \\
        = & S(\mu^{i-1}\alpha_{T}(x_{n_{1}}\otimes \cdots \otimes x_{n_{r}})\otimes [x_{n_{r+1}}, x_{n_{r+2}}]_{\mathfrak{g}}\otimes \alpha_{T} (\cdots \otimes x_{n_{i+1}})) \\
        + & S(\mu^{i-1}\alpha_{T}(x_{n_{1}}\otimes \cdots)\otimes [x_{n_{r}}, x_{n_{r+2}}]_{\mathfrak{g}}\otimes \alpha_{T}(x_{n_{i+1}}\otimes \cdots \otimes x_{n_{i+1}})) \\
        + & S(\mu^{i-1}\alpha_{T}(x_{n_{1}}\otimes \cdots \otimes x_{n_{r+2}})\otimes [x_{n_{r}}, x_{n_{r+1}}]_{\mathfrak{g}}\otimes \alpha_{T} (\cdots \otimes x_{n_{i+1}}))
      \end{align*}
      using the hom-Jacobi identity. Thus we can obtain
      \begin{align*}
        u= &  S(x_{n_{1}}\otimes  \cdots \otimes x_{n_{r+2}}\otimes x_{n_{r+1}}\otimes x_{n_{r}} \otimes \cdots \otimes x_{n_{i+1}}) \\
        + & S(\mu^{i-1}\alpha_{T}(x_{n_{1}}\otimes \cdots \otimes x_{n_{r}})\otimes [x_{n_{r+1}}, x_{n_{r+2}}]_{\mathfrak{g}}\otimes \alpha_{T}(\cdots \otimes x_{n_{i+1}})) \\
        + & S(\mu^{i-1}\alpha_{T}(x_{n_{1}}\otimes \cdots)\otimes [x_{n_{r}}, x_{n_{r+2}}]_{\mathfrak{g}}\otimes \alpha_{T}(x_{n_{i+1}}\otimes \cdots \otimes x_{n_{i+1}})) \\
        + & S(\mu^{i-1}\alpha_{T}(x_{n_{1}}\otimes \cdots \otimes x_{n_{r+2}})\otimes [x_{n_{r}}, x_{n_{r+1}}]_{\mathfrak{g}}\otimes \alpha_{T}(\cdots \otimes x_{n_{i+1}})) \\
        = & v.
      \end{align*}
Now that $u=v$ in either cases, let $\mathfrak{x}\in J_{\mathfrak{g},\beta}\cap kW$. Then $S(\mathfrak{x})=\mathfrak{x}$ and $S(\mathfrak{x})=0$. Therefore $\mathfrak{x}=0$ and we get that $J_{\mathfrak{g},\beta}\cap kW=0$ which completes the proof.
\end{proof}

We are now ready to prove the Poincare-Birkhoff-Witt theorem for involutive color hom-Lie algebras in the second part of the next theorem.
\begin{thm}
Let $k$ be a field whose characteristic is not 2. Let $\mathfrak{g}:=(\mathfrak{g},[,]_{\mathfrak{g}},\beta_{\mathfrak{g}})$ be an involutive hom-Lie color algebra on $k$. Let $\theta:T(\mathfrak{g})\rightarrow T(\mathfrak{g})$ be as described in \eqref{teta}. Let $I$ be the hom-ideal of $T(\mathfrak{g})$ generated by the commutators defined in Theorem \ref{I}. Let $J$ be as defined in \eqref{J}. Then there is a well-ordered basis $X$ of $\mathfrak{g}$ such that for
$$W=W_{X}=\{ x_{i_{1}}\otimes \cdots \otimes x_{i_{n}}|i_{1}\geq \cdots \geq i_{n}, n\geq0\},$$
the following statements hold.
\begin{itemize}
  \item [(i)] $T(\mathfrak{g})=J\oplus k H,$
  \item [(ii)] $\theta(W)$ is a basis of $U(\mathfrak{g})$.
\end{itemize}
\end{thm}
\begin{proof}
\begin{itemize}
  \item [(i)] Let $\mathfrak{g}_{+}$ and $\mathfrak{g}_{-}$ be the eigenspaces of 1 and -1 of $\mathfrak{g}$, respectively. Then we have the decomposition
      $$\mathfrak{g}=\mathfrak{g}_{+}\oplus \mathfrak{g}_{-}.$$

      Let $B_{+}$ and $B_{-}$ be some basis for $\mathfrak{g}_{+}$ and $\mathfrak{g}_{-}$, respectively. There are two different cases:\\
      If the cardinality of $B_{+}$ is more than the cardinality of $B_{-}$, fix an injection $\iota:B_{-}\rightarrow B_{+}$. Then the following set is a basis of $\mathfrak{g}$
      $$X:=\{\iota(x)+x,\iota(x)-x|x\in B_{-}\} \cup(B_{+}\setminus B_{-}),$$
      and $\beta_{\mathfrak{g}}(X)=X$. Let $W$ be defined with a well order on $X$ and choose $\mu=1$ in \eqref{J}. Then we have
      $$T(\mathfrak{g})= J_{\mathfrak{g},\beta,1}\oplus kW.$$
      If the cardinality of $B_{+}$ is not more than the cardinality of $B_{+}$, it suffices to assume $\gamma:=-\beta_{\mathfrak{g}}$ and take $\mu=-1$ as it is like the last case.
  \item [(ii)] Follows directly from Theorem.\ref{dec} and $(i)$.
\end{itemize}
\end{proof}
\section{Hom-Lie Superalgebras Case}

The study of hom-Lie superalgebras has been widely in the center of interest these last years. The motivation came from the generalization of Lie superalgebras, or in some cases, the generalization of hom-Lie algebras. However, here we deal with them as a special case of color hom-Lie algebras. One can simply put $\Gamma=\mathds{Z}_{2}$ in Definition \eqref{HLCAD} and define $\varepsilon$ in such a way that $\varepsilon(x,y)=(-1)^{|x||y|}$ to get the following definition \cite{COHS, COH,LSGradedquasiLiealg,YN,CCH,spl}.

\begin{defn}
A hom-Lie superalgebra is a triple $(\mathfrak{g},[,],\alpha)$ consisting of a superspace $\mathfrak{g}$, a bilinear map $[,]:\mathfrak{g}\times \mathfrak{g} \rightarrow \mathfrak{g}$ and a superspace homomorphism $\alpha:\mathfrak{g}\rightarrow \mathfrak{g}$, both of them of degree zero satisfying
\begin{itemize}
  \item [1.]$[x,y]=-(-1)^{|x||y|}[y,x],$
  \item [2.]$(-1)^{|x||z|}[\alpha(x),[y,z]]+(-1)^{|y||x|}[\alpha(y),[z,x]]+(-1)^{|z||y|}[\alpha(z),[x,y]]=0,$
\end{itemize}
for all homogeneous elements $x,y,z \in \mathfrak{g}$.
\end{defn}
In particular, one can easily recall from the first section, the notions of a multiplicative hom-Lie superalgebra, a morphism of hom-Lie superalgebras and a hom-associative superalgebra. Moreover, a hom-associative superalgebra or a hom-Lie superalgebra is said to be involutive if $\alpha^{2}=id$.

Again, as in the previous cases, a hom-associative superalgebra $(V,\mu,\alpha)$ gives a hom-Lie superalgebra by antisymmetrization. We denote this hom-Lie superalgebra again by $(A,[,]_{A},\beta_{A})$, where $\beta_{A}=\alpha$, $[x,y]_{A}=xy-yx$, for all $x,y\in A$.

Let $(V,\alpha_{V})$ be an involutive hom-module. A free involutive hom-associative color algebra on $V$ is an involutive hom-associative super algebra $(F(V),*,\alpha_{F})$ together with a morphism of hom-modules $j_{V}:(V,\alpha_{V})\rightarrow (F(V),\alpha_{F})$ with the property that, for any involutive hom-associative superalgebra $(A,.,\alpha_{A})$ together with a morphism $f:(V,\alpha_{V})\rightarrow (A,\alpha_{A})$ of hom-modules, there is a unique morphism $\overline{f}:(F(V),*,\alpha_{F})\rightarrow (A,.,\alpha_{A})$ of hom-associative superalgebras such that $f=\overline{f}\circ j_{V}$.

The definition of the universal enveloping algebra as can be predicted is just a modification of the Definition \ref{uni}.
\begin{defn}\label{uniS}
Let $(\mathfrak{g},[,],\alpha)$ be a hom-Lie superalgebra. A universal enveloping hom-associative superalgebra of $\mathfrak{g}$ is a hom associative superalgebra
$$U(\mathfrak{g}):=(U(\mathfrak{g}),\mu_{U},\alpha_{U}),$$
together with a morphism $\varphi_{\mathfrak{g}}:\mathfrak{g}\rightarrow U(\mathfrak{g})$ of hom-Lie superalgebras such that for any hom-associative superalgebra $(A,\mu,\alpha_{A})$ and any hom-Lie superalgebra morphism $\phi:(\mathfrak{g},[,]_{\mathfrak{g}},\beta_{\mathfrak{g}})$, there exists a unique morphism $\bar{\phi}:U(\mathfrak{g}\rightarrow A)$ of hom associative superalgebras such that $\bar{\phi}\circ \varphi_{\mathfrak{g}}=\phi$.
\end{defn}

The following lemma shows an easy way to construct the universal algebra when we have an involutive hom-Lie superalgebra.

\begin{lem}\label{3tayiS}
Let $(\mathfrak{g},[,]_{\mathfrak{g}},\beta_{\mathfrak{g}})$ be an inovolutive hom-Lie superalgebra.
\begin{itemize}
  \item [(i)] Let $(A,\cdot,\alpha_{A})$ be a hom-associative algebra, $f:(\mathfrak{g},[,]_{\mathfrak{g}}, \beta_{\mathfrak{g}})\rightarrow (A,[,]_{A},\beta_{A})$ be a morphism of hom-Lie superalgebras and $B$ be the hom-associative subsuperalgebra of $A$ generated by $f(\mathfrak{g})$. Then $B$ is involutive.
  \item [(ii)] The universal enveloping hom-associative algebra $(U(\mathfrak{g}),\varphi_{\mathfrak{g}})$ of $(\mathfrak{g},[,]_{\mathfrak{g}},\beta_{\mathfrak{g}})$ is involutive.
  \item [(iii)] In order to verify the universal property of $(U(\mathfrak{g}),\varphi_{\mathfrak{g}})$ in Definition \ref{uniS}, we only need to consider involutive hom-associative algebras $A:=(A,\cdot_{A},\alpha_{A})$.
\end{itemize}
\end{lem}

We can now, give the construction of the universal enveloping hom-associative superalgebra of an involutive hom-Lie superalgebra.

\begin{thm}\label{IS}
Let $\mathfrak{g}:= (\mathfrak{g}, [,]_{\mathfrak{g}},\beta_{\mathfrak{g}})$ be an involutive hom-Lie superalgebra. Let $$T(\mathfrak{g}):=(T(\mathfrak{g}),\odot,\alpha_{T})$$ be the free hom-associative algebra on the hom-module underlying $\mathfrak{g}$. Let $I$ be the hom-ideal of $T(\mathfrak{g})$ generated by the set
\begin{equation}\label{21S}
  \{a\otimes b -(-1)^{|a||b|} b\otimes a -[a,b]\}
\end{equation}
and let
$$U(\mathfrak{g})=\frac{T(\mathfrak{g})}{I}$$
be the quotient hom-associative algebra. Let $\psi$ be the composition of the natural inclusion $i:\mathfrak{g}\rightarrow T(\mathfrak{g})$ with the quotient map $\pi:T(\mathfrak{g})\rightarrow U(\mathfrak{g})$. Then $(U(\mathfrak{g}),\psi)$ is a universal enveloping hom-associative algebra of $\mathfrak{g}$. Also, the universal
enveloping hom-associative algebra of $\mathfrak{g}$ is unique up to isomorphism.
\end{thm}

We can also have a Poincare-Birkhoff-Witt like type theorem for involutive hom-Lie superalgebras as a special case of color hom-Lie algebras, i.e. if $\mathfrak{g}$ is a Lie superalgebra with an ordered basis $X=\{x_{n}|n\in H\}$ where $H$ is a well totally ordered set, let $I$ be the ideal of the free associative algebra $T(\mathfrak{g})$ on $\mathfrak{g}$ which was given in Theorem \ref{IS}, so that $U(\mathfrak{g})$ is the universal enveloping algebra of $\mathfrak{g}$. We also suppose that $\mathfrak{g}$ is an involutive hom-Lie superalgerba with a basis $X=\{x_{n}|n\in \omega \}$ for a well ordered set $\omega$. Next theorem which is a combination of theorems in the previous section, will help us give the Poincare-Birkhoff-Witt theorem for involutive hom-Lie superalgebras.

\begin{thm}
Let $\mathfrak{g}:= (\mathfrak{g},[,]_{\mathfrak{g}},\beta_{\mathfrak{g}})$ be an involutive hom-Lie superalgebra such that $\beta_{\mathfrak{g}}(X)=X$. Let $\theta:T(\mathfrak{g})\rightarrow T(\mathfrak{g})$ be as defined in \eqref{teta}, $I$ be the hom-ideal of $T(\mathfrak{g})$ as defined in Theorem \ref{IS}, $W$ be like the one defined in \eqref{tfW} and let $\mu\in k$ be given. Moreover, let
\begin{multline} \label{JS} J_{\mu}:=
\sum_{n,m\geq0}\sum_{\substack {\mathfrak{a}\in \mathfrak{g}^{\otimes n}\\ \mathfrak{b}\in \mathfrak{g}^{\otimes m}}}\sum_{a,b\in \mathfrak{g}}(\mathfrak{a}\otimes (a\otimes b -(-1)^{|a||b|}b\otimes a)\otimes \mathfrak{b} \\
- \mu^{n+m} \alpha_{T}(\mathfrak{a})\otimes [a,b]_{\mathfrak{g}}\otimes \alpha_{T}(\mathfrak{b})),
\end{multline}
Then
\begin{itemize}
  \item [(i)]
$
I=  \sum_{n,m\geq0}\sum_{a,b\in \mathfrak{g}}(\mathfrak{g}^{\otimes n} \odot (a\otimes b-(-1)^{|a||b|}b\otimes a -[a,b]_{\mathfrak{g}}) )\odot \mathfrak{g}^{\otimes m}.
$
  \item [(ii)] $\theta(I)=J$
  \item [(iii)] We can have the linear decomposition $T(\mathfrak{g})=J_{\mu}\oplus k W$.
\end{itemize}
\end{thm}
Finally, we give the Poincare-Birkhoff-Witt theorem for involutive hom-Lie superalgebras.
\begin{thm}
Let $k$ be a field whose characteristic is not 2. Let $\mathfrak{g}:=(\mathfrak{g},[,]_{\mathfrak{g}},\beta_{\mathfrak{g}})$ be an involutive hom-Lie superalgebra on $k$. Let $\theta:T(\mathfrak{g})\rightarrow T(\mathfrak{g})$ be as described in \eqref{teta}. Let $I$ be the hom-ideal of $T(\mathfrak{g})$ generated by the commutators defined in Theorem \ref{IS}. Let $J$ be as defined in \eqref{JS}. Then there is a well-ordered basis $X$ of $\mathfrak{g}$ such that for
$$W=W_{X}=\{ x_{i_{1}}\otimes \cdots \otimes x_{i_{n}}|i_{1}\geq \cdots \geq i_{n}, n\geq0\}$$
the following statements hold.
\begin{itemize}
  \item [(i)] $T(\mathfrak{g})=J\oplus k H,$
  \item [(ii)] $\theta(W)$ is a basis of $U(\mathfrak{g})$.
\end{itemize}
\end{thm}
\section{Acknowledgments}
This paper is supported by grant no. 92grd1m82582 of Shiraz university, Shiraz, Iran. A. Armakan is grateful to Mathematics and Applied Mathematics research environment MAM, Division of Applied Mathematics, School of Education, Culture and Communication at M{\"a}lardalen University, V{\"a}ster{\aa}s, Sweden for creating excellent research environment during his visit from September 2016 to March 2017.


\begin{thebibliography}{999}
\bibitem{COH} K. Abdaoui, F. Ammar, A. Makhlouf, Constructions and cohomology of hom-Lie color algebras, Comm. Algebra, 43, (2015), 4581-4612.
\bibitem{Kerner7} V. Abramov, B. Le Roy, R. Kerner, Hypersymmetry: a $\mathbb{Z}_3$-graded generalization of supersymmetry, J. Math. Phys., \textbf{38} (3), (1997), 1650-1669.
\bibitem{AizawaSaito}
N. Aizawa, H. Sato, $q$-deformation of the Virasoro algebra with central extension, Phys. Lett. B \textbf{256}, (1991), 185-190. Hiroshima University preprint, preprint HUPD-9012 (1990).
\bibitem{AmmarMakhloufHomLieSupAlg2010}
F. Ammar, A. Makhlouf, Hom-Lie Superalgebras and Hom-Lie admissible Superalgebras, J. Algebra, \textbf{324}, 7, (2010), 1513-1528.
\bibitem{AmEjMakhCohDefHomalgJLT2011}
F. Ammar, Z. Ejbehi, A. Makhlouf, Cohomology and Deformations of Hom-algebras, J. Lie Theory 21(2011), 4, 813-836, Zbl 1237.17003, MR2917693.
\bibitem{AmmarMabroukMakhloufCohomnaryHNLalg2011}
F. Ammar, S. Mabrouk, A. Makhlouf, Representations and Cohomology of $n$-ary multiplicative Hom-Nambu-Lie algebras, J. Geometry and Physics, (2011),  DOI: 10.1016/j.geomphys.2011.04.022
\bibitem{COHS} F. Ammar, A. Makhlouf and N. Saadaoui, Chomology of hom-Lie superalgebras and q-deformed Witt superalgebra, Czechoslovak Math. J. 68 (2013), 721-761.
\bibitem{ArnlindKituoniMakhloufSilv3aryCohom} J. Arnlind, A. Kitouni, A. Makhlouf, S. Silvestrov, Structure and Cohomology of 3-Lie algebras induced by Lie algebras, In "Algebra, Geometry and Mathematical Physics", Springer proceedings in Mathematics and Statistics, vol 85, (2014).
\bibitem{ArnlindMakhloufSilvnaryHomLieNambuJMP2011} J. Arnlind, A. Makhlouf, S. Silvestrov, Construction of n-Lie algebras and n-ary Hom-Nambu-Lie algebras, Journal of Mathematical Physics, 52 (12), (2011), 123502.
\bibitem{ArnlindMakhloufSilvTernHomNambuJMP2010} J. Arnlind, A. Makhlouf, S. Silvestrov, Ternary Hom-Nambu-Lie algebras induced by Hom-Lie algebras, Journal of Mathematical Physics, 51 (4), (2010), 043515, 11 pp.
\bibitem{BahturinMikhPetrZaicevIDLSbk92} Y. A. Bahturin, A. A. Mikhalev, V. M. Petrogradsky, M. V. Zaicev, Infinite Dimensional Lie Superalgebras, Walter de Gruyter, Berlin, (1992).
\bibitem{BM} S. Benayadi, A. Makhlouf, Hom-Lie algebras with symmetric invariant nondegenerate bilinear forms, J. Geom. Phys. 76, (2014), 38-60.
\bibitem{spl} A. Calderon, J. S. Delgado, On the structure of split Lie color algebras, Linear Algebra Appl. 436, (2012), 307-315.
\bibitem{CCH} Y. Cao, L. Chen, On split regular hom-Lie color algebras, Comm. Algebra 40, (2012), 575-592.
\bibitem{CasasInsuaPachecoUncenextHLie} J. M. Casas, M. A. Insua, N. Pacheco,  \emph{On universal central extensions of hom-Lie algebras}, Hacet. J. Math. Stat. 44(2015), 2, 277-288, Zbl 06477457, MR3381108.
\bibitem{ChaiElinPop} M. Chaichian, D. Ellinas, Z. Popowicz, Quantum conformal algebra with central extension, Phys. Lett. B \textbf{248}, (1990), 95-99.
\bibitem{ChaiIsLukPopPresn}
M. Chaichian, A. P. Isaev, J. Lukierski, Z. Popowic, P. Pre\v{s}najder, $q$-deformations of Virasoro algebra and conformal dimensions, Phys. Lett. B \textbf{262} (1), (1991), 32-38.
\bibitem{ChaiKuLuk}  M. Chaichian, P. Kulish, J. Lukierski, $q$-deformed Jacobi identity, $q$-oscillators and $q$-deformed infinite-dimensional algebras, Phys. Lett. B \textbf{237}, (1990), 401-406.
\bibitem{ChaiPopPres} M. Chaichian, Z. Popowicz, P. Pre\v{s}najder, $q$-Virasoro algebra and its relation to the $q$-deformed KdV system, Phys. Lett. B \textbf{249}, (1990), 63--65.
\bibitem{ChenPetitOystaeyenCOHCHLA} C. W. Chen, T. Petit, F. Van Oystaeyen, Note on cohomology of color Hopf and Lie algebras, J. Algebra, 299 (2006), 419-442
\bibitem{CurtrZachos1} T. L. Curtright, C. K. Zachos, Deforming maps for quantum algebras, Phys. Lett. B \textbf{243}, (1990), 237-244.
\bibitem{DamKu} E. V. Damaskinsky, P. P. Kulish, Deformed oscillators and their applications (in Russian), Zap. Nauch. Semin. LOMI 189 (1991), 37-74. [Engl. transl. in J. Sov. Math., 62 (1992), 2963-2986.
\bibitem{DaskaloyannisGendefVir} C. Daskaloyannis, Generalized deformed Virasoro algebras, Modern Phys. Lett. A \textbf{7} no. 9, (1992), 809--816.
\bibitem{GuoZhZheUEPBWHLieA} L. Guo, B. Zhang and S. Zheng, \emph{Universal enveloping algebras and Poincare-Birkhoff-Witt theorem for involutive hom-Lie algebras}, arXiv:1607.05973 [math.QA], (2016)
\bibitem{HLS} J. T. Hartwig, D. Larsson, S. D. Silvestrov, Deformations of Lie algebras using $\sigma$-derivations, J. of Algebra \textbf{295}, (2006), 314-361. Preprint in Mathematical Sciences 2003:32, LUTFMA-5036-2003, Centre for Mathematical Sciences, Department of Mathematics, Lund Institute of Technology, 52 pp. (2003).
\bibitem{HeMaSiUnAlHomAss}
L. Hellstr{\"o}m, A. Makhlouf, S. D. Silvestrov, Universal Algebra Applied to Hom-Associative Algebras, and More. In: Makhlouf A., Paal E., Silvestrov S., Stolin A. (eds) Algebra, Geometry and Mathematical Physics. Springer Proceedings in Mathematics and Statistics, vol 85. Springer, Berlin, Heidelberg, (2014), 157-199.
\bibitem{HelSil-bookC} L. Hellstr{\"o}m, S. D. Silvestrov, Commuting Elements in $q$-Deformed Heisenberg Algebras, World Scientific, Singapore, 2000, 256 pp. (ISBN: 981-02-4403-7).
\bibitem{Hu} N. Hu, $q$-Witt algebras, $q$-Lie algebras, $q$-holomorph structure and representations,  Algebra Colloq. \textbf{6}, no. 1, (1999), 51-70.
\bibitem{KharchenkoQLTbook2015} V. Kharchenko, Quantum Lie Theory - A Multilinear Approach, Lecture Notes in Mathematics,  2150, Springer International Publishing, 2015.
\bibitem{Kassel92} C. Kassel, Cyclic homology of differential operators, the virasoro algebra and a $q$-analogue, Comm. Math. Phys. 146(2), (1992), 343-356.
\bibitem{Kerner} R. Kerner, Ternary algebraic structures and their applications in physics, in the ``Proc. BTLP 23rd International Colloquium on Group Theoretical Methods in Physics'', (2000). arXiv:math-ph/0011023v1
\bibitem{Kerner2} R. Kerner, $\mathbb{Z}_3$-graded algebras and non-commutative gauge theories, dans le livre "Spinors, Twistors, Clifford Algebras and Quantum Deformations", Eds. Z. Oziewicz, B. Jancewicz, A. Borowiec, pp. 349-357, Kluwer Academic Publishers (1993).
\bibitem{Kerner4} R. Kerner, The cubic chessboard: Geometry and physics, Classical Quantum Gravity \textbf{14}, (1997), A203-A225.
\bibitem{Kerner6} R. Kerner, L. Vainerman, On special classes of $n$-algebras, J. Math. Phys., \textbf{37} (5), (1996), 2553-2565.
\bibitem{Lecomte} P. Lecomte, On some Sequence of graded Lie algebras associated to manifolds, Ann. Global Analysis Geom. 12 (1994), 183-192.
\bibitem{LS1} D. Larsson, S. D. Silvestrov, Quasi-Hom-Lie algebras, Central Extensions and $2$-cocycle-like identities, J. Algebra \textbf{288}, (2005), 321-344.
Preprints in Mathematical Sciences  2004:3, LUTFMA-5038-2004, Centre for Mathematical Sciences, Department of Mathematics, Lund Institute of Technology, Lund University, 2004.
\bibitem{LS2} D. Larsson,  S. D. Silvestrov, Quasi-Lie algebras. In
"Noncommutative Geometry and Representation Theory in Mathematical Physics". Contemp. Math., 391, Amer. Math. Soc., Providence, RI, (2005), 241--248. Preprints in Mathematical Sciences 2004:30, LUTFMA-5049-2004, Centre for Mathematical Sciences, Department of Mathematics, Lund Institute of Technology, Lund University, 2004.
\bibitem{LarssonSigSilvJGLTA2008} D. Larsson, G. Sigurdsson, S. D. Silvestrov, Quasi-Lie deformations on the algebra $\mathbb{F}[t]/(t^N)$, J. Gen. Lie Theory Appl.  2  (2008), 201-205.
\bibitem{LSGradedquasiLiealg} D. Larsson, S. D. Silvestrov, Graded quasi-Lie agebras, Czechoslovak J. Phys. \textbf{55}, (2005), 1473-1478.
\bibitem{LS3}
D. Larsson, S. D. Silvestrov, Quasi-deformations of $sl_2(\mathbb{F})$ using twisted derivations, Comm. in Algebra \textbf{35}, (2007), 4303-4318.
\bibitem{LiuKQuantumCentExt} K. Q. Liu, Quantum central extensions, C. R. Math. Rep. Acad. Sci. Canada \textbf{13} (4), (1991), 135-140.
\bibitem{LiuKQCharQuantWittAlg} K. Q. Liu, Characterizations of the Quantum Witt Algebra, Lett. Math. Phys. \textbf{24} (4), (1992), 257-265.
\bibitem{LiuKQPhDthesis} K. Q. Liu, The Quantum Witt Algebra and Quantization of Some Modules over Witt Algebra, PhD Thesis, Department of Mathematics, University of Alberta, Edmonton, Canada (1992).
\bibitem{MAK} A. Makhlouf, \emph{Paradigm of nonassociative hom-algebras and hom-superalgebras}, Proceedings of Jordan Structures in Algebra and Analysis Meeting, (2010), 145-177.
\bibitem{MS} A. Makhlouf, S. D. Silvestrov, Hom-algebra structures, J. Gen. Lie Theory Appl. Vol \textbf{2} (2), (2008),  51--64. (Zbl 1184.17002, MR2399415). Preprints in Mathematical Sciences  2006:10, LUTFMA-5074-2006, Centre for Mathematical Sciences, Department of Mathematics, Lund Institute of Technology, Lund University, 2006.
\bibitem{HomHopf} A. Makhlouf, S. D. Silvestrov, Hom-Lie admissible Hom-coalgebras and Hom-Hopf algebras, In "Generalized Lie theory in Mathematics, Physics and Beyond. S. Silvestrov, E. Paal, V. Abramov, A. Stolin, Editors". Springer-Verlag, Berlin, Heidelberg, Chapter 17, 189--206, (2009). Preprints in Mathematical Sciences, Lund University, Centre for Mathematical Sciences, Centrum Scientiarum Mathematicarum (2007:25) LUTFMA-5091-2007 and in arXiv:0709.2413 [math.RA] (2007).
\bibitem{HomAlgHomCoalg}  A. Makhlouf, S. D. Silvestrov, Hom-Algebras and Hom-Coalgebras, J. Algebra Appl., Vol. \textbf{9}, No. 4 (2010) 553-589. Preprints in Mathematical Sciences, Lund University, Centre for Mathematical Sciences, Centrum Scientiarum Mathematicarum, (2008:19) LUTFMA-5103-2008 and in arXiv:0811.0400 [math.RA] (2008).
\bibitem{HomDeform} A. Makhlouf, S. Silvestrov, Notes on 1-parameter formal deformations of Hom-associative and Hom-Lie algebras, Forum Math. 22 (2010), no. 4, 715-739. (Zbl 1201.17012, MR2661446). Preprints in Mathematical Sciences, Lund University, Centre for Mathematical Sciences, Centrum Scientiarum Mathematicarum, (2007:31) LUTFMA-5095-2007. arXiv:0712.3130v1 [math.RA] (2007).
\bibitem{MikhZolotykhCALSbk95} A. A. Mikhalev, A. A. Zolotykh, Combinatorial Aspects of Lie Superalgebras, CRC Press, 1995.
\bibitem{PiontkovskiSilvestrovC3dCLA} D. Piontkovski, S. D. Silvestrov, Cohomology of 3-dimensional color Lie algebras, J. Algebra, Vol 316, Issue 2, (2007), 499-513.
\bibitem{RichardSilvestrovJA2008} L. Richard, S. D. Silvestrov, Quasi-Lie structure of $\sigma$-derivations of $\mathbb{C}[t^{\pm1}]$, J. Algebra  319  (2008),  no. 3, 1285-1304.
\bibitem{RichardSilvestrovGLTbdSpringer2009} L. Richard, S. Silvestrov, A note on quasi-Lie and Hom-Lie structures of $\sigma$-derivations of $\mathbb{C}[z^{\pm1}_1,\dots,z^{\pm1}_n]$, In "Generalized Lie theory in mathematics, physics and beyond", Springer, Berlin, (2009), 257-262.
\bibitem{ScheunertGLA} M. Scheunert, Generalized Lie algebras, J. Math. Phys. 20, 4 (1979), 712-720.
\bibitem{ScheunertGTC} M. Scheunert, Graded tensor calculus, J. Math. Phys. 24, (1983), 2658-2670.
\bibitem{ScheunertCOH2} M. Scheunert, Introduction to the cohomology of Lie superalgebras and some applications, Res. Exp. Math., 25, (2002), 77-107.
\bibitem{ScheunertZHA} M. Scheunert, R. B. Zhang, Cohomology of Lie superalgebras and their generalizations, J. Math. Phys. 39, (1998), 5024-5061.
\bibitem{SigSilvGLTbdSpringer2009} G. Sigurdsson, S. Silvestrov, Lie color and Hom-Lie algebras of Witt type and their central extensions, In "Generalized Lie theory in mathematics, physics and beyond", Springer, Berlin, (2009), 247-255
\bibitem{Czech:witt} G. Sigurdsson and S. Silvestrov, Graded quasi-Lie algebras of Witt type, Czech. J. Phys. (2006) 56: 1287-1291.
\bibitem{SilvestrovParadigmQLieQhomLie2007} S. Silvestrov, Paradigm of quasi-Lie and quasi-Hom-Lie algebras and quasi-deformations. In "New techniques in Hopf algebras and graded ring theory", K. Vlaam. Acad. Belgie Wet. Kunsten (KVAB), Brussels, (2007), 165-177.
\bibitem{SR} Y. Sheng, Representations of hom-Lie algebras, Algebr. Represent. Theory 15, (2012), 1081-1098, Zbl 1294.17001, MR2994017.
\bibitem{SC} Y. Sheng, D. Chen, Hom-Lie 2-algebras, J. Algebra 376, (2013), 174-195, Zbl 1281.17034, MR3003723.
\bibitem{SB} Y. Sheng, C. Bai, A new approach to hom-Lie bialgebras, J. Algebra 399, (2014),232-250.
\bibitem{SX} Y. Sheng and Z. Xiong, On hom-Lie algebras, Linear Multilinear Algebra 63(2015), 12, 2379-2395, Zbl 06519840, MR3402544.
\bibitem{Yau:EnvLieAlg} D. Yau, Enveloping algebra of Hom-Lie algebras, J. Gen. Lie Theory Appl. \textbf{2} (2), (2008), 95-108. (Zbl 1214.17001, MR2399418)
\bibitem{Yau:HomolHomLie} D. Yau, Hom-algebras  and homology, J. Lie Theory, \textbf{19}, (2009), 409--421.
\bibitem{Yau:HomBial} D. Yau, Hom-bialgebras and comodule algebras, Int. Electron. J. Algebra, \textbf{8}, (2010), 45-64.
\bibitem{YH} D. Yau, Hom-Yang-Baxter equation, Hom-Lie algebras, and quasi-triangular bialgebras, J. Phys. A: Math. Theor. 42(2009), 165202, Zbl 1179.17001, MR2539278.
\bibitem{YN} L. Yuan, Hom-Lie color algebra structures, Comm. Algebra 40, (2012), 575-592.
\bibitem{DER} J. Zhou, L. Chen, Y. Ma, Generalized derivations of hom-Lie superalgebras, Acta Math. Sinica (Chin. Ser.) 58, (2014), 3737-3751.
\end{thebibliography}
\end{document}